\documentclass[12pt, twoside]{article}
\usepackage{times}
\usepackage{bm}
\usepackage{graphics}
\usepackage{theorem}
\usepackage{graphicx}
\usepackage{amsmath}
\usepackage{latexsym}
\usepackage{amssymb,mathrsfs}
\usepackage[flushmargin]{footmisc}

\theorembodyfont{\itshape}
\newtheorem{theorem}{Theorem}[section]
\newtheorem{lemma}{Lemma}[section]

\theorembodyfont{\rmfamily}
\newtheorem{definition}{Definition}[section]{\bf}{\rm}
\newtheorem{assumpt}{Assumption}[section]{\bf}{\rm}

\newtheorem{rem}{Remark}[section]{\itshape}{\rmfamily}

\newenvironment{proof}{\noindent{\it Proof.~~}}{\qed\medskip}

\makeatletter
\def\eqnarray{\stepcounter{equation}\let\@currentlabel=\theequation
\global\@eqnswtrue
\global\@eqcnt\z@\tabskip\@centering\let\\=\@eqncr
$$\halign to \displaywidth\bgroup\@eqnsel\hskip\@centering
  $\displaystyle\tabskip\z@{##}$&\global\@eqcnt\@ne 
  \hfil$\;{##}\;$\hfil
  &\global\@eqcnt\tw@ $\displaystyle\tabskip\z@{##}$\hfil 
   \tabskip\@centering&\llap{##}\tabskip\z@\cr}
\makeatother
\makeatletter
    \renewcommand{\theequation}{%
    \thesection.\arabic{equation}}
    \@addtoreset{equation}{section}
  \makeatother

\newcommand{\dm}{\displaystyle}
\newcommand{\qed}{\hspace*{\fill}$\Box$}
\newcommand{\lleft}{\!\!\left}
\newcommand{\vc}{\bm}
\def\svc#1{\mbox{\boldmath $\scriptstyle #1$}}

\newcommand{\wt}{\widetilde}

\newcommand{\vmax}{\vee}

\def\trunc#1{{}_{(n)}#1}

\newcommand{\sfBI}{\mathsf{BI}}
\newcommand{\sfBM}{\mathsf{BM}}
\newcommand{\EE}{\mathsf{E}}
\newcommand{\PP}{\mathsf{P}}

\newcommand{\calE}{\mathcal{E}}

\newcommand{\bbD}{\mathbb{D}}
\newcommand{\bbF}{\mathbb{F}}

\newcommand{\bbN}{\mathbb{N}}

\newcommand{\bbZ}{\mathbb{Z}}

\newcommand{\rmd}{{\rm d}}
\newcommand{\rme}{{\rm e}}

\renewcommand{\labelenumi}{(\roman{enumi})}
\newcommand{\dd}[1]{\if#11 1\!\!1 
\else {\if#1C I\!\!\!C
\else {\if#1G I\!\!\!G 
\else {\if#1J J\!\!\!J 
\else {\if#1S S\!\!\!S
\else {\if#1Z Z\!\!\!Z
\else {\if#1Q O\!\!\!\!Q
\else I\!\!#1
\fi}
\fi}
\fi}
\fi} 
\fi} 
\fi} 
\fi} 

\makeatother

\begin{document}\thispagestyle{plain} 

\hfill

{\Large{\bf
\begin{center}
Error bounds for augmented truncations of discrete-time block-monotone 
Markov chains under subgeometric drift conditions%
\footnote[1]{This paper is published in {\it SIAM Journal on Matrix Analysis and Applications}, vol.\ 37, no.\ 3, pp.\ 877--910, 2016. 
}
%
%
\end{center}
}
}

\begin{center}
{
Hiroyuki Masuyama%
\footnote[2]{E-mail: masuyama@sys.i.kyoto-u.ac.jp}
}

\medskip

{\small
Department of Systems
Science, Graduate School of Informatics, Kyoto University\\
Kyoto 606-8501, Japan
}

\bigskip
\medskip

{\small
\textbf{Abstract}

\medskip

\begin{tabular}{p{0.85\textwidth}}
This paper studies the last-column-block-augmented northwest-corner
truncation (LC-block-augmented truncation, for short) of discrete-time
block-monotone Markov chains under subgeometric drift conditions.  The
main result of this paper is to present an upper bound for the total
variation distance between the stationary probability vectors of a
block-monotone Markov chain and its LC-block-augmented truncation. The
main result is extended to Markov chains that themselves may not be
block monotone but are block-wise dominated by block-monotone Markov
chains satisfying modified drift conditions. Finally, as an
application of the obtained results, the GI/G/1-type Markov chain is
considered.
\end{tabular}
}
\end{center}

\begin{center}
\begin{tabular}{p{0.90\textwidth}}
{\small

\medskip

}
\end{tabular}

\end{center}

\section{Introduction}\label{introduction}

This paper considers the truncation approximation of discrete-time
block-structured Markov chains characterized by an infinite number of
block matrices.  The typical examples of such Markov chains are
M/G/1-, GI/M/1- and GI/G/1-type Markov chains and level-dependent
quasi-birth-and-death processes (LD-QBDs) (see, e.g., \cite{He14} and
the references therein) and they appear as the queue length processes
of various semi-Markovian queues (see, e.g., \cite{Seng90}).

It is a challenging problem to obtain the stationary probability
vectors of block-structured Markov chains characterized by an infinite
number of block matrices. To solve this problem, we often use the
augmented northwest-corner truncation (augmented truncation, for
short). More specifically, we form a finite stochastic matrix by
augmenting, in some way, the (finite) northwest-corner truncation of
the transition probability matrix and then adopt the stationary
probability vector of the resulting finite stochastic matrix as an
approximation to that of the original Markov chain.

Although there are infinitely many variations of such augmented
truncation, this paper focuses on the last-column-block-augmented
northwest-corner truncation (LC-block-augmented truncation, for short)
because it is proved that the LC-block-augmentation is the best (in a
certain sense) among all the {\it block augmentations} if they are
applied to the northwest-corner truncations of block-monotone
transition probability matrices (see \cite[Theorem 3.6]{Li00} and
\cite[Theorem 4.1]{Masu15-LAA}). Note that block monotonicity is an
extension of (classical) monotonicity to block-structured Markov
chains \cite[Definition 2.5]{Li00}.

The main purpose of this paper is to estimate the error of the
stationary probability vector obtained by the LC-block-augmented
truncation.  There are some related studies on the truncation of
Markov chains.  Tweedie~\cite{Twee98} presented a
total-variation-distance error bound for the stationary probability
vector of the last-column-augmented northwest-corner truncation of a
monotone and geometrically ergodic Markov chain in discrete time. More
precisely, Tweedie~\cite{Twee98}'s bound is an upper bound for the
total variation distance between the stationary probability vectors of
the original Markov chain and its last-column-augmented
northwest-corner truncation.  Hart and Tweedie~\cite{Hart12} and
Liu~\cite{Liu15} discussed the convergence of the stationary
probability vectors of the augmented truncations of continuous-time
Markov chains with monotonicity and/or exponential ergodicity.
Masuyama~\cite{Masu15,Masu15-LAA} extended Tweedie~\cite{Twee98}'s
result to block-monotone Markov chains in discrete and continuous
time.

Without the monotonicity of Markov chains, Herv\'{e} and
Ledoux~\cite{Herv14} derived a total-variation-distance error bound
for the stationary probability vector of the last-column-augmented
northwest-corner truncation of a geometrically ergodic Markov chain in
discrete time.  Zeifman et al.~\cite{Zeif14a} considered a periodic
and exponentially weakly ergodic non-time-homogeneous birth-and-death
process in continuous time, and they presented a
total-variation-distance error bound for the periodic stationary
distribution obtained by the truncation of the state space (see also \cite{Zeif14b,Zeif12}).

Basically, all the existing results mentioned above assume geometric
ergodicity (exponential ergodicity in continuous time) and
thus they are not applicable to Markov chains with subgeometric
ergodicity (including polynomial ergodicity). For example, reflected
Markov additive processes and GI/G/1-type Markov chains with the
heavy-tailed asymptotics \cite{Jarn03,Mao14}, which typically arise
from BMAP/GI/1 queues with subexponential service times and/or batch
sizes \cite{Masu09,Masu15-ANOR}. As far as we know, there are no
studies on the error estimation of the augmented truncation of Markov
chains with subgeometric ergodicity, except for Liu~\cite{Liu10}'s
work. Liu~\cite{Liu10} derived a total-variation-distance error bound
for the stationary probability vector of the last-column-augmented
northwest-corner truncation of a monotone and polynomially ergodic
Markov chain in discrete time.

In this paper, we consider a block-monotone Markov chain under the
drift condition proposed by Douc et al.~\cite{Douc04}, which covers
polynomial ergodicity and other types of subgeometric ergodicity. We
first derive an upper bound for the total variation distance between
the stationary probability vectors of the original Markov chain and
its LC-block-augmented truncation, which is the main result of this
paper. We also present a similar bound in the case where the original
Markov chain itself may not be block monotone but is block-wise
dominated by a block-monotone Markov chain satisfying a modified drift
condition with a larger tolerance for boundary exceptions. The
modified drift condition strengthens the applicability of the obtained
bound. Finally, we provide a detailed procedure for establishing our
error bounds for the LC-block-augmented truncations of block-monotone
GI/G/1-type Markov chains.

The rest of this paper is divided into five
sections. Section~\ref{sec-preliminaries} provides preliminary results
on block-monotone stochastic matrices. Section~\ref{sec-main-results}
presents the main result of this paper. Section~\ref{sec-extension}
contains the extension of the main result to possibly
non-block-monotone Markov chains that are block-wise dominated by
block-monotone Markov chains satisfying the modified drift
condition. Section~\ref{sec-applications} considers the application of
the extended result to block-monotone GI/G/1-type Markov chains.
Finally, Section~\ref{sec-remarks} provides concluding remarks.

\section{Preliminaries}\label{sec-preliminaries}

Let $\{(X_{\nu},J_{\nu});\nu\in\bbZ_+\}$ denote a Markov chain with
state space $\bbF=\bbZ_+ \times \bbD$, where $\bbZ_+=\{0,1,2,\dots\}$
and $\bbD = \{1,2,\dots,d\} \subset \bbN:=\{1,2,\dots\}$.  Let
$\vc{P}:=(p(k,i;\ell,j))_{(k,i),(\ell,j)\in\bbF}$ denote the
transition probability matrix of the Markov chain
$\{(X_{\nu},J_{\nu})\}$, i.e.,
\[
p(k,i;\ell,j)
=\PP(X_{\nu+1} = \ell, J_{\nu+1} = j \mid X_{\nu} = k, J_{\nu} = i),
\qquad (k,i;\ell,j) \in \bbF^2,
\]
where $(k,i;\ell,j)$ represents ordered pair $((k,i),(\ell,j))$. Note
here that $\vc{P}$ is row stochastic (stochastic, for short,
hereafter), i.e., $\vc{P}\vc{e} = \vc{e}$, where $\vc{e}$ denotes a
column vector of 1's of an appropriate order.

For $n \in \bbN$, let
$\trunc{\vc{P}}_n:=(\trunc{p}_n(k,i;\ell,j))_{(k,i),(\ell,j)\in\bbF}$
denote a stochastic matrix such that, for $i,j\in\bbD$,
\begin{equation}
\trunc{p}_n(k,i;\ell,j)
= 
\left\{
\begin{array}{ll}
p(k,i;\ell,j), & k \in \bbZ_+,\ \ell = 0,1,\dots,n-1,
\\
\dm\sum_{m=n}^{\infty} p(k,i;m,j), & k \in \bbZ_+,\ \ell = n,
\\
0, & \mbox{otherwise}.
\end{array}
\right.
\label{def-trunc{p}_n(k,i;l,j)}
\end{equation}
By definition, $\vc{P}$ and $\trunc{\vc{P}}_n$ can be partitioned into
block matrices with size $d$. Furthermore,
(\ref{def-trunc{p}_n(k,i;l,j)}) implies that $\trunc{\vc{P}}_n$ is in
the following form:
\begin{equation}
\trunc{\vc{P}}_n
= \bordermatrix{
               	& \bbF^{\leqslant n} 	&   \bbF \setminus \bbF^{\leqslant n}    
\cr
\bbF^{\leqslant n} 	& 
\trunc{\vc{P}}_n^{\leqslant n}	& 
\vc{O}
\cr
\bbF \setminus \bbF^{\leqslant n} 		& 
\trunc{\vc{P}}_n^{\ast} &	
\vc{O} 
},
\label{add-eqn-22}
\end{equation}
where $\bbF^{\leqslant n} = \{0,1,\dots,n\} \times \bbD$ and $\vc{O}$
denotes the zero matrix of an appropriate order. Equation
(\ref{add-eqn-22}) shows that the sub-state space $\bbF \setminus
\bbF^{\leqslant n}$ of $\trunc{\vc{P}}_n$ is transient and thus the
submatrix $\trunc{\vc{P}}_n^{\ast}$ of $\trunc{\vc{P}}_n$ does not
have any contribution to the stationary probability vector of
$\trunc{\vc{P}}_n$. Therefore, we call the whole matrix
$\trunc{\vc{P}}_n$ the {\it last-column-block-augmented
  (LC-block-augmented) northwest-corner truncation} ({\it
  LC-block-augmented truncation}, for short).  The LC-block-augmented
truncation $\trunc{\vc{P}}_n$ is also called the {\it
  last-column-block-augmented first-$n$-block-column truncation} in
\cite{Masu15}.  It should be noted that, since
$\trunc{\vc{P}}_n^{\leqslant n}$ in (\ref{add-eqn-22}) is a finite
stochastic matrix, the LC-block-augmented truncation
$\trunc{\vc{P}}_n$ always has at least one stationary probability
vector (see, e.g., \cite[Chapter 3, Theorems 3.1 and 3.3]{Brem99}),
which is denoted by
$\trunc{\vc{\pi}}_n:=(\trunc{\pi}_n(k,i))_{(k,i)\in \bbF}$.

We now assume that $\vc{P}$ is irreducible and positive recurrent. We
then define $\vc{\pi}:=(\pi(k,i))_{(k,i)\in \bbF}$ as the unique
stationary probability vector of $\vc{P}$.  We also assume, unless
otherwise stated, that $\vc{P}$ is block monotone with block sized $d$
(see \cite[Definition 1.1]{Masu15}), i.e.,
\begin{equation}
\sum_{m=\ell}^{\infty}\vc{P}(k;m) \le \sum_{m=\ell}^{\infty}\vc{P}(k+1;m),
\qquad k,\ell \in \bbZ_+,
\label{defn-P-BM_d}
\end{equation}
where $\vc{P}(k;\ell) := (p(k,i;\ell,j))_{i,j\in\bbD}$ is the
$(k,\ell)$th block of $\vc{P}$.  To shorten the statements on block
monotonicity, let $\sfBM_d$ denote the set of block-monotone
stochastic matrices with block size $d$. It then follows from $\vc{P}
\in \sfBM_d$ that a stochastic matrix $\sum_{m=0}^{\infty}\vc{P}(k;m)$
is constant with $k \in \bbZ_+$ (see \cite[Proposition
  1.1]{Masu15}). Furthermore, since $\vc{P}$ is irreducible, the
stochastic matrix $\vc{\Psi}:=\sum_{m=0}^{\infty}\vc{P}(k;m)$ is
irreducible and thus has the unique stationary probability vector,
denoted by $\vc{\varpi}:=(\varpi(i))_{i\in\bbD}$.

Finally, we introduce some symbols and definitions related to block
monotonicity. Let
\[
\vc{T}_d
= \left(
\begin{array}{ccccc}
\vc{I}_d & \vc{O} & \vc{O} & \vc{O} & \cdots
\\
\vc{I}_d & \vc{I}_d & \vc{O} & \vc{O} & \cdots
\\
\vc{I}_d & \vc{I}_d & \vc{I}_d & \vc{O} & \cdots
\\
\vc{I}_d & \vc{I}_d & \vc{I}_d & \vc{I}_d & \cdots
\\
\vdots      & \vdots      & \vdots      & \vdots      & \ddots
\end{array}
\right),~~
\vc{T}_d^{-1}
= \left(
\begin{array}{ccccc}
\vc{I}_d & \vc{O} & \vc{O} & \vc{O} & \cdots
\\
-\vc{I}_d & \vc{I}_d & \vc{O} & \vc{O} & \cdots
\\
\vc{O} & -\vc{I}_d & \vc{I}_d & \vc{O} & \cdots
\\
\vc{O} & \vc{O} & -\vc{I}_d & \vc{I}_d & \cdots
\\
\vdots      & \vdots      & \vdots      & \vdots      & \ddots
\end{array}
\right),
\] 
where $\vc{I}_d$ denotes the $d \times d$ identity matrix (we write
$\vc{I}$ for the identity matrix whose order is clear from the
context). Note here that (\ref{defn-P-BM_d}) is equivalent to
\[
\vc{T}_d^{-1} \vc{P}\vc{T}_d \ge \vc{O}.
\]
\begin{definition}\label{defn-BI}
A column vector $\vc{f}=(f(k,i))_{(k,i)\in\bbF}$ with block size $d$
is said to be block increasing if $\vc{T}_d^{-1}\vc{f} \ge \vc{0}$,
i.e., $f(k,i) \le f(k+1,i)$ for all $(k,i) \in \bbZ_+ \times \bbD$. We
denote by $\sfBI_d$ the set of block-increasing column vectors with
block size $d$.
\end{definition}

\medskip

\begin{definition}\label{defn-block-wise-dominance}
A probability vector $\vc{\mu}:=(\mu(k,i))_{(k,i)\in\bbF}$ with block
size $d$ is said to be block-wise dominated by a probability vector
$\vc{\eta}:=(\eta(k,i))_{(k,i)\in\bbF}$ (denoted by $\vc{\mu} \prec_d
\vc{\eta}$) if $\vc{\mu}\vc{T}_d \le \vc{\eta}\vc{T}_d$.  Similarly, a
stochastic matrix
$\vc{P}_1:=(p_1(k,i;\ell,j))_{(k,i),(\ell,j)\in\bbF}$ with block size
$d$ is said to be block-wise dominated by a stochastic matrix
$\vc{P}_2:=(p_2(k,i;\ell,j))_{(k,i),(\ell,j)\in\bbF}$ (denoted by
$\vc{P}_1 \prec_d \vc{P}_2$) if $\vc{P}_1\vc{T}_d \le
\vc{P}_2\vc{T}_d$.
\end{definition}

\medskip

It is known that if $\vc{P} \in \sfBM_d$ then $\trunc{\vc{P}}_n
\prec_d \vc{P} $ and thus $\trunc{\vc{\pi}}_n \prec_d \vc{\pi}$ (see
\cite[Proposition~2.3]{Masu15}), which leads to
\begin{equation}
\sum_{k=0}^{\infty} \trunc{\pi}_n(k,i) 
= \sum_{k=0}^{\infty} \pi(k,i) = \varpi(i),\qquad i \in \bbD.
\label{eqn-varpi(i)}
\end{equation}

\section{Main result}\label{sec-main-results}

This section presents an upper bound for $\| \trunc{\vc{\pi}}_n -
\vc{\pi} \|$, where $\| \cdot \|$ denotes the total variation
distance, i.e.,
\[
\| \trunc{\vc{\pi}}_n - \vc{\pi} \|
= \sum_{(k,i)\in\bbF} | \trunc{\pi}_n(k,i) - \pi(k,i) |.
\]
Let $\vc{p}^m(k,i) := (p^m(k,i;\ell,j))_{(\ell,j)\in\bbF}$ and
$\trunc{\vc{p}}_n^m(k,i)
:=(\trunc{p}_n^m(k,i;\ell,j))_{(\ell,j)\in\bbF}$ denote probability
vectors such that $p^m(k,i;\ell,j)$ and $\trunc{p}_n^m(k,i;\ell,j)$
represent the $(k,i;\ell,j)$th elements of $\vc{P}^m$ and
$(\trunc{\vc{P}}_n)^m$, respectively. For any function
$\varphi(\cdot,\cdot)$ on $\bbF$, let $\varphi(k,\vc{\varpi}) =
\sum_{i\in\bbD} \varpi(i)\varphi(k,i)$ for $k \in \bbZ_+$. We then
have
\begin{eqnarray}
\left\| \trunc{\vc{\pi}}_n - \vc{\pi} \right\|
&\le& \left\| \vc{p}^m(0,\vc{\varpi}) - \vc{\pi} \right\|
+ \left\| \trunc{\vc{p}}_n^m(0,\vc{\varpi}) - \trunc{\vc{\pi}}_n \right\|
\nonumber
\\
&& {} \quad 
+ \left\| \trunc{\vc{p}}_n^m(0,\vc{\varpi}) - \vc{p}^m(0,\vc{\varpi}) \right\|.
\label{add-eqn-14}
\end{eqnarray}
It also follows from the second last inequality in the proof of
\cite[Theorem 3.1]{Masu15} that
\begin{equation}
\left\| \trunc{\vc{p}}_n^m(0,\vc{\varpi}) - \vc{p}^m(0,\vc{\varpi}) \right\|
\le 2m\sum_{i\in\bbD}\trunc{\pi}_n(n,i).
\label{estimation-01}
\end{equation}

To estimate the first and second terms on the right hand side of
(\ref{add-eqn-14}), we assume the subgeometric drift condition
proposed in \cite{Douc04}, which is described in
Assumption~\ref{assumpt-poly} below. For the description of the drift
condition, let $\vc{1}_K=(1_K(k,i))_{(k,i)\in\bbF}$, $K\in \bbZ_+$
denote a column vector such that
\[
1_K(k,i)
=\left\{
\begin{array}{ll}
1, & (k,i) \in \bbF^{\leqslant K},
\\
0, & (k,i) \in \bbF \setminus \bbF^{\leqslant K}.
\end{array}
\right.
\]
In addition, for any scalar-valued function $\theta$ on
$(-\infty,\infty)$ and any real-valued column vector $\vc{a}:=(a(i))$,
let $\theta \circ \vc{a} = (\theta \circ a(i))$.

\medskip

\begin{assumpt}[{}{\cite[Condition ${\bf D}(\phi,V,C)$]{Douc04}}]\label{assumpt-poly}
There exist a constant $b \in (0,\infty)$, a column vector
$\vc{v}=(v(k,i))_{(k,i)\in\bbF} \in \sfBI_d$ with $\vc{v} \ge \vc{e}$,
and a nondecreasing differentiable concave function $\phi:[1,\infty)
  \to (0,\infty)$ with $\lim_{t\to\infty}\phi'(t) = 0$ such that
\begin{equation}
\vc{P}\vc{v} \le \vc{v} - \phi \circ \vc{v} + b\vc{1}_0.
\label{ineqn-Pv-02}
\end{equation}

\end{assumpt}

\medskip

\begin{rem}\label{rem-drift-cond}
If $\lim_{t\to\infty}\phi'(t) = c$ for some $c > 0$, then
Assumption~\ref{assumpt-poly} is reduced to the geometric drift
condition (see \cite[Remark 1]{Douc04} for the details): There exist
$b \in (0,\infty)$, $\gamma \in (0,1)$ and a column vector $\vc{v} \in
\sfBI_d$ with $\vc{v} \ge \vc{e}$ such that
\begin{equation*}
\vc{P}\vc{v} \le \gamma \vc{v} + b\vc{1}_0,
\end{equation*}
which is assumed in the
related studies \cite{Herv14,Masu15,Twee98}.
\end{rem}

\medskip

Under Assumption~\ref{assumpt-poly}, the irreducible stochastic matrix
$\vc{P}$ is subgeometrically ergodic if $\vc{P}$ is aperiodic
\cite[Proposition 2.5]{Douc04}. However, we do not necessarily assume
the aperiodicity of $\vc{P}$.

We now introduce some symbols according to
\cite[Section~2]{Douc04}. Let $H_{\phi}$ denote a function on
$[1,\infty)$ such that
\begin{equation}
H_{\phi}(x) = \int_1^x {\rmd y \over \phi(y)},\qquad x \ge 1.
\label{defn-H_{phi}}
\end{equation}
Clearly, $H_{\phi}$ is an increasing differentiable concave function,
and $\lim_{x\to\infty}H_{\phi}(x) = \infty$ due to the concavity of
$\phi$ (see \cite[section 2]{Douc04}). Thus, the inverse
$H_{\phi}^{-1}:[0,\infty) \to [1,\infty)$ of $H_{\phi}$ is
    well-defined and $\lim_{x\to\infty}H_{\phi}^{-1}(x) = \infty$.
    Furthermore, it follows from (\ref{defn-H_{phi}}) that
    $H_{\phi}^{-1}$ is an increasing differentiable function and
\begin{equation}
r_{\phi}(x)
:=
(H_{\phi}^{-1})'(x) = \phi \circ H_{\phi}^{-1}(x),
\qquad x \ge 0.
\label{defn-r_{phi}}
\end{equation}
The function $r_{\phi}$ is nondecreasing because $\phi$ is
nondecreasing and $H_{\phi}^{-1}$ is increasing. In addition, it
follows from \cite[Proposition~2.1]{Douc04} that $r_{\phi}$ is
log-concave. For convenience, we define $r_{\phi}(x) = 0$ for $x < 0$.

In what follows, we present two lemmas and the main result of this
paper.

\medskip

\begin{lemma}\label{lem-bound-moment-tau}
Consider the Markov chain $\{(X_{\nu},J_{\nu});\nu\in\bbZ_+\}$ with
state space $\bbF$ and transition probability matrix $\vc{P}$, and let
$\tau_0^+ = \inf\{\nu \in \bbN; X_{\nu} = 0\}$. If $\vc{P} \in
\sfBM_d$, $\vc{P}$ is irreducible and Assumption~\ref{assumpt-poly}
holds, then
\begin{eqnarray}
\EE_{(k,i)}[r_{\phi}(\tau_0^+-1)]
&\le& v(k \vee 1,i), \qquad (k,i) \in \bbF, 
\label{bound-moment-tau(a)}
\\
\EE_{\svc{\pi}}[r_{\phi}(\tau_0^+-1)]
&\le& v(k \vee 1,\vc{\varpi}), \qquad  ~~\,k \in \bbZ_+, 
\label{bound-moment-tau(b)}
\end{eqnarray}
where $x \vee y=\max(x,y)$ and
\begin{eqnarray*}
\EE_{\svc{\pi}}[\,\cdot\,] 
&=& \sum_{(k,i)\in\bbF}\pi(k,i)\EE_{(k,i)}[\,\cdot\,],
\\
\EE_{(k,i)}[\,\cdot\,] &=& \EE[\,\cdot \mid X_0=k,J_0=i],\qquad (k,i) \in \bbF.
\end{eqnarray*}

\end{lemma}

\medskip

\begin{proof}
We first prove (\ref{bound-moment-tau(a)}).
Note that
\begin{equation}
\EE_{(k,i)} \lleft[ r_{\phi}(\tau_0^+-1) \right]
\le \EE_{(k,i)} \lleft[ \sum_{n=0}^{\tau_0^+-1} r_{\phi}(n)\right],
\qquad (k,i) \in \bbF.
\label{eqn-35a}
\end{equation}
It then follows from \cite[Proposition~2.2]{Douc04} that
\begin{equation}
\EE_{(k,i)}
\lleft[ \sum_{n=0}^{\tau_0^+ - 1} r_{\phi}(n)\right] \le v(k,i),
\qquad k \in \bbN,\ i \in \bbD.
\label{add-eqn-0518-01}
\end{equation}
Furthermore, 
\begin{equation}
\EE_{(0,i)} 
\lleft[ \sum_{n=0}^{\tau_0^+ - 1} r_{\phi}(n) \right] 
\le 
\EE_{(k,i)} \lleft[ \sum_{n=0}^{\tau_0^+ - 1} r_{\phi}(n) \right],
\qquad k \in \bbN,\ i \in \bbD,
\label{add-eqn-0518-02}
\end{equation}
which follows from the pathwise-ordered property \cite[Lemma
  A.1]{Masu15} of the block-monotone Markov chain
$\{(X_{\nu},J_{\nu})\}$.  Combining (\ref{add-eqn-0518-01}) and
(\ref{add-eqn-0518-02}), we have
\begin{equation}
\EE_{(k,i)} 
\lleft[ \sum_{n=0}^{\tau_0^+-1} r_{\phi}(n)\right] \le v(k \vee 1,i), 
\qquad (k,i) \in \bbF.
\label{eqn-36}
\end{equation}
Substituting (\ref{eqn-36}) into (\ref{eqn-35a}), we obtain
(\ref{bound-moment-tau(a)}).

Next, we prove (\ref{bound-moment-tau(b)}). It follows from
\cite[Theorem~10.4.9]{Meyn09} that, for any function
$\varphi(\cdot,\cdot)$ on $\bbF$ such that
$\sum_{(k,i)\in\bbF}\pi(k,i)|\varphi(k,i)| < \infty$,
\begin{equation}
\sum_{(k,i)\in\bbF}\pi(k,i)\varphi(k,i)
= \sum_{i\in\bbD}\pi(0,i)
\EE_{(0,i)} 
\lleft[ \sum_{\nu=0}^{\tau_0^+-1} \varphi(X_{\nu},J_{\nu}) \right].
\label{add-eqn-0517-01}
\end{equation}
Note here that $\EE_{\svc{\pi}}[r_{\phi}(\tau_0^+-1)] =
\EE_{\svc{\pi}} \lleft[ \sum_{n=0}^{\tau_0^+-1} \Delta r_{\phi}(n)
  \right]$, where $\Delta r_{\phi}(n) = r_{\phi}(n) - r_{\phi}(n-1)$
for $n\in\bbZ_+$.  Thus, letting
$\varphi(k,i)=\EE_{(k,i)}[\sum_{n=0}^{\tau_0^+-1} \Delta r_{\phi}(n)]$
in (\ref{add-eqn-0517-01}), we obtain
\begin{eqnarray}
\EE_{\svc{\pi}}[r_{\phi}(\tau_0^+-1)]
&=& 
\EE_{\svc{\pi}} 
\lleft[ \sum_{n=0}^{\tau_0^+-1} \Delta r_{\phi}(n) \right]
\nonumber
\\
&=& \sum_{i\in\bbD}\pi(0,i)
\EE_{(0,i)} 
\lleft[ \sum_{\nu=0}^{\tau_0^+-1}
\EE_{(X_{\nu},J_{\nu})}
\lleft[ 
\sum_{n=\nu}^{\tau_0^+(\nu)-1} \Delta r_{\phi}(n)
\right]
\right], \qquad
\label{add-eqn-0517-02}
\end{eqnarray}
where $\tau_0^+(\nu) = \inf\{n \ge \nu+1;X_n = 0\}$. Changing the
order of summation on the right hand side of (\ref{add-eqn-0517-02}),
we have
\begin{eqnarray}
\qquad
\lefteqn{
\EE_{\svc{\pi}}[r_{\phi}(\tau_0^+-1)]
}
\quad &&
\nonumber
\\
&=& \sum_{i\in\bbD}\pi(0,i)
\EE_{(0,i)} 
\lleft[ \sum_{n=0}^{\tau_0^+-1} 
\sum_{\nu=0}^n \Delta r_{\phi}(n)
\right]
\nonumber
\\
&=& \sum_{i\in\bbD}\pi(0,i)\EE_{(0,i)} 
\lleft[ \sum_{n=0}^{\tau_0^+-1} r_{\phi}(n) \right]
\le \sum_{i\in\bbD} \varpi(i)\EE_{(k,i)} 
\lleft[ \sum_{n=0}^{\tau_0^+-1} r_{\phi}(n)\right],
\quad k \in \bbZ_+,
\label{eqn-35b}
\end{eqnarray}
where the last inequality follows from (\ref{add-eqn-0518-02}) and
$\pi(0,i) \le \varpi(i)$ for $i \in \bbD$ (see
(\ref{eqn-varpi(i)})). Applying (\ref{eqn-36}) to (\ref{eqn-35b})
yields (\ref{bound-moment-tau(b)}).
\end{proof}

\medskip

Using Lemma~\ref{lem-bound-moment-tau}, we prove the following lemma.

\medskip

\begin{lemma}\label{lem-bounds-p^m-subgeo}
Suppose that $\vc{P} \in \sfBM_d$ and $\vc{P}$ is irreducible. If
Assumption~\ref{assumpt-poly} holds, then, for all $k \in \bbZ_+$ and
$m\in\bbN$,
\begin{eqnarray}
\left\|  \vc{p}^m(k,\vc{\varpi}) - \vc{\pi} \right\|
&\le& {4 \over r_{\phi}(m-1)} v(k \vee 1,\vc{\varpi}),
\label{eqn-17a}
\\
\left\|
\trunc{\vc{p}}_n^m(k,\vc{\varpi}) - \trunc{\vc{\pi}}_n  
\right\|
&\le& {4 \over r_{\phi}(m-1)} v(k \vee 1,\vc{\varpi}),
\qquad n \in \bbN,
\label{eqn-17b}
\end{eqnarray}
where the function $r_{\phi}$ is given in (\ref{defn-r_{phi}}).
\end{lemma}

\medskip

\begin{proof}
We first prove (\ref{eqn-17a}). Let $\tau_0 =
\inf\{n\in\bbZ_+;X_n=0\}$ and $\PP_{(k,i)}(\,\cdot\,) =
\PP(\,\cdot\mid X_0=k,J_0=i)$ for $(k,i) \in \bbF$. Following the
derivation of \cite[Eq.~(3.22)]{Masu15} and replacing $\vc{v}$ by
$\vc{e}$, we have
\begin{eqnarray}
\left\| \vc{p}^m(k,\vc{\varpi}) - \vc{\pi}  \right\|
&\le& 2\sum_{i\in\bbD}\varpi(i)\PP_{(k,i)}(\tau_0 > m)
 + 2\sum_{(\ell,j)\in\bbF}\pi(\ell,j)\PP_{(\ell,j)}(\tau_0 > m)
\nonumber
\\ 
&\le& 2\sum_{i\in\bbD}\varpi(i)\PP_{(k,i)}(\tau_0^+ > m)
+ 2\sum_{(\ell,j)\in\bbF}\pi(\ell,j)\PP_{(\ell,j)}(\tau_0^+ > m),
\label{eqn-09}
\end{eqnarray}
where the second inequality follows from the fact that
$\PP_{(\ell,j)}(\tau_0 > m) \le \PP_{(\ell,j)}(\tau_0^+ > m)$ for $m
\in \bbN$ and $(\ell,j)\in\bbF$.  Using Markov's inequality, we have
\[
\PP_{(\ell,j)}(\tau_0^+ > m) 
\le {1 \over r_{\phi}(m-1)}\EE_{(\ell,j)}[r_{\phi}(\tau_0^+-1)],
\qquad  m \in \bbN,\ (\ell,j)\in\bbF.
\]
Thus
\begin{eqnarray}
\qquad\quad
\left\| \vc{p}^m(k,\vc{\varpi}) - \vc{\pi} \right\|
&\le& {2 \over r_{\phi}(m-1)}
\left(
\sum_{i\in\bbD}\varpi(i)\EE_{(k,i)}[r_{\phi}(\tau_0^+-1)]
+ \EE_{\svc{\pi}}[r_{\phi}(\tau_0^+-1)]
\right).
\label{eqn-18}
\end{eqnarray}
Applying Lemma~\ref{lem-bound-moment-tau} to (\ref{eqn-18}) yields
(\ref{eqn-17a}). In addition, since $\trunc{\vc{P}}_n \prec_d \vc{P}$
and $\vc{v} \in \sfBI_d$, it follows from (\ref{ineqn-Pv-02}) and
\cite[Remark~2.1]{Masu15} that
\begin{equation}
\trunc{\vc{P}}_n\vc{v}
\le \vc{P}\vc{v} \le \vc{v} - \phi \circ \vc{v} + b \vc{1}_0.
\label{ineqn-trunc{p}_n*v}
\end{equation}
Therefore, (\ref{eqn-17b}) can be proved in the same way as that of
(\ref{eqn-17a}). The details are omitted.
\end{proof}

\medskip

Theorem~\ref{thm-main-subgeo} below presents upper bounds for $\|
\trunc{\vc{\pi}}_n - \vc{\pi} \|$, which are the main result of this
paper.

\medskip

\begin{theorem}\label{thm-main-subgeo}
Suppose that $\vc{P} \in \sfBM_d$ and $\vc{P}$ is irreducible. If
Assumption~\ref{assumpt-poly} holds, then, for all $n,m \in \bbN$,
\begin{eqnarray}
\left\| \trunc{\vc{\pi}}_n - \vc{\pi} \right\|
&\le&
{8 \over r_{\phi}(m-1)} v(1,\vc{\varpi})
+ 2m \sum_{i\in\bbD}\trunc{\pi}_n(n,i),
\label{error-bound-subgeo-01}
\\
\left\| \trunc{\vc{\pi}}_n - \vc{\pi} \right\|
&\le&
{8 \over r_{\phi}(m-1)} v(1,\vc{\varpi})
+ 2mb \sum_{i\in\bbD} {1 \over \phi \circ v(n,i)}, 
\label{error-bound-subgeo-02}
\end{eqnarray}
where the function $r_{\phi}$ is given in (\ref{defn-r_{phi}}).
\end{theorem}

\medskip

\begin{proof}
 Substituting (\ref{estimation-01}), (\ref{eqn-17a}) and
 (\ref{eqn-17b}) into (\ref{add-eqn-14}), we obtain
 (\ref{error-bound-subgeo-01}).  Furthermore, pre-multiplying both
 sides of (\ref{ineqn-trunc{p}_n*v}) by $\trunc{\vc{\pi}}_n$, we have
\[
\sum_{(k,i)\in\bbF}\trunc{\pi}_n(k,i) \cdot \phi \circ v(k,i) 
\le  b,
\]
which leads to
\begin{equation}
\trunc{\pi}_n(n,i) \le {b \over \phi \circ v(n,i)}. 
\label{add-eqn-0519-01}
\end{equation}
Combining (\ref{error-bound-subgeo-01}) and (\ref{add-eqn-0519-01}) 
yields (\ref{error-bound-subgeo-02}). 
\end{proof}

\medskip

\begin{rem}\label{rem-main-thm}
Suppose that $\lim_{t\to\infty}\phi(t) = \infty$ and
$\lim_{k\to\infty}v(k,i) = \infty$ for all $i\in\bbD$. It then holds
that $\lim_{n\to\infty}\phi \circ v(n,i) = \infty$ for
$i\in\bbD$. Thus, it follows from (\ref{add-eqn-0519-01}) that
$\lim_{n\to\infty}\trunc{\pi}_n(n,i) = 0$ for $i\in\bbD$. Recall here
that $\lim_{x\to\infty}H_{\phi}^{-1}(x) = \infty$. Applying this and
$\lim_{t\to\infty}\phi(t) = \infty$ to (\ref{defn-r_{phi}}) yields
$\lim_{x\to\infty}r_{\phi}(x) = \infty$. Therefore, we can reduce the
error bounds (\ref{error-bound-subgeo-01}) and
(\ref{error-bound-subgeo-02}) to any desired value less than two by
fixing $m \in \bbN$ sufficient large and then choosing $n \in \bbN$
large such that the bounds take the desired value.
\end{rem}

\section{Extension of the main result}\label{sec-extension}

In this section, we extend Theorem~\ref{thm-main-subgeo} to the case
where $\vc{P}$ may not be block monotone but is block-wise dominated
by a block-monotone stochastic matrix
$\widetilde{\vc{P}}=(\widetilde{p}(k,i;\ell,j))_{(k,i),(\ell,j)\in\bbF}$. Note
here that $\widetilde{\vc{P}} \in \sfBM_d$ is allowed to be equal to
$\vc{P}$ and thus $\vc{P} \in \sfBM_d$.

Let
$\widetilde{\vc{P}}^m:=(\widetilde{p}^m(k,i;\ell,j))_{(k,i),(l,j)\in\bbF}$,
$m\in\bbN$ denote the $m$th power of $\widetilde{\vc{P}}$. For $k,\ell
\in \bbZ_+$, let
$\widetilde{\vc{P}}^m(k;\ell):=(\widetilde{p}^m(k,i;\ell,j))_{i,j\in\bbD}$
denote the $(k,\ell)$th block of $\widetilde{\vc{P}}^m$. We then have
the following result.

\medskip

\begin{theorem}\label{thm-extended-geo2}
Suppose that (i) $\vc{P} \prec_d \widetilde{\vc{P}} \in \sfBM_d$ and
$\widetilde{\vc{P}}$ is irreducible; and that (ii) there exist a
constant $b \in (0, \infty)$, a column vector
$\vc{v}=(v(k,i))_{(k,i)\in\bbF} \in \sfBI_d$ with $\vc{v} \ge \vc{e}$
and a nondecreasing differentiable concave function $\phi:[1,\infty)
  \to (0,\infty)$ with $\lim_{t\to\infty}\phi'(t) = 0$ such that, for
  some $M\in\bbN$ and $K \in \bbZ_+$,
\begin{eqnarray}
\widetilde{\vc{P}}^M \vc{v}
&\le& \vc{v} - \phi \circ \vc{v}
+ b\vc{1}_K,
\label{ineqn-hat{P}-widetilde{v}}
\\
\widetilde{\vc{P}}^M(K;0)\vc{e} &>& \vc{0}.
\label{ineqn-hat{P}(K,0)e}
\end{eqnarray}
Let $B$ denote a positive constant such that
\begin{eqnarray}
B\cdot\widetilde{\vc{P}}^M(K;0)\vc{e} \ge b\vc{e}.
\label{defn-B}
\end{eqnarray}
Under these conditions, the following bound holds for all
$m,n\in\bbN$.
\begin{eqnarray}
\left\| \trunc{\vc{\pi}}_n - \vc{\pi} \right\|
&\le&
{
8\{ c_{\phi,B}(1) \}^{-1} 
\over r_{\phi} \circ c_{\phi,B}(m-1)
} 
\{ v(1,\vc{\varpi}) + B\}
+ 2mMb \sum_{i\in\bbD}{1 \over \phi \circ v(n,i)},
\label{bound-pi-04}
\end{eqnarray}
where 
\begin{equation}
c_{\phi,B}(x)
= {\phi(1) \over \phi(B+1)} x,
\qquad x \ge 0.
\label{defn-c_{phi,B}}
\end{equation}
In addition, if $K=0$, then, for all $m,n\in\bbN$, 
\begin{eqnarray}
\left\| \trunc{\vc{\pi}}_n - \vc{\pi} \right\|
&\le&
{
8 
\over r_{\phi}(m-1)
} 
v(1,\vc{\varpi})
+ 2mMb \sum_{i\in\bbD}{1 \over \phi \circ v(n,i)},
\label{bound-pi-K=0}
\end{eqnarray}
which holds without (\ref{ineqn-hat{P}(K,0)e}).
\end{theorem}

\medskip

\begin{rem}
Although the condition (\ref{ineqn-hat{P}(K,0)e}) ensures the
existence of a constant $B\in (0,\infty)$ satisfying (\ref{defn-B}),
it may seem that the condition (\ref{ineqn-hat{P}(K,0)e}) weakens the
applicability of Theorem~\ref{thm-extended-geo2}. However, that is not
necessarily the case.  To verify this, we assume that all the
conditions of Theorem~\ref{thm-extended-geo2} are satisfied, except
for (\ref{ineqn-hat{P}(K,0)e}), and that $\widetilde{\vc{P}}$ is
aperiodic. Note here that the aperiodicity of $\widetilde{\vc{P}}$
does not make any restriction because $(\vc{I}+\vc{P})/2$ and
$(\vc{I}+\widetilde{\vc{P}})/2$ are aperiodic and
\begin{eqnarray*}
\vc{\pi}(\vc{I}+\vc{P})/2 &=& \vc{\pi},
\qquad \qquad \quad 
\widetilde{\vc{\pi}}(\vc{I}+\widetilde{\vc{P}})/2 = \widetilde{\vc{\pi}},
\\
(\vc{I}+\vc{P})/2 &\prec_d& (\vc{I}+\widetilde{\vc{P}})/2,
\qquad (\vc{I}+\widetilde{\vc{P}})/2 \in \sfBM_d.
\end{eqnarray*}
We also assume the following mild condition. 
\begin{equation*}
\lim_{k\to\infty}
\phi \circ v(k,i) = \infty\quad \mbox{for all $i \in \bbD$}.
\end{equation*}
Under these conditions, let $\{b_m;m\in\bbN\}$ and $\{K_m;m\in\bbN\}$
denote sequences such that $b_1 = b$, $K_1=K$ and, for $m =2,3,\dots$,
\begin{eqnarray*}
b_m &=& \left\{
b+b_{m-1} - (1-\lambda)\min_{i\in\bbD}\phi \circ v(0,i)
\right\} \vmax 0,
\\
K_m &=& 
\inf\left\{
k\in\bbN; (b + b_{m-1}) \vc{e} 
- (1-\lambda) \phi \circ \vc{v}(k) \le \vc{0}
\right\} - 1,
\end{eqnarray*}
where $\lambda \in (0,1)$ is fixed arbitrarily.
It then follows from (\ref{ineqn-hat{P}-widetilde{v}}) that
\begin{eqnarray*}
\widetilde{\vc{P}}^{mM} \vc{v}
&\le& \vc{v} - \lambda \phi \circ \vc{v}
+ b_m \vc{1}_{K_m},\qquad m \in \bbN.
\end{eqnarray*}
In addition, the aperiodicity of $\widetilde{\vc{P}}$ implies that for
any $m \in\bbN$ there exists some $\nu_0(K_m)\in\bbN$ such that
$\widetilde{\vc{P}}^{\nu}(K_m;0)\vc{e} > \vc{0}$ for all $\nu \ge
\nu_0(K_m)$ (see, e.g., \cite[Chapter 2, Theorem 4.3]{Brem99}). As a
result, if there exists some $m\in\bbN$ such that $\nu_0(K_m) \le mM$,
then we can obtain a bound similar to (\ref{bound-pi-04}) by using
Theorem~\ref{thm-extended-geo2}.
\end{rem}

\medskip

{\it Proof of Theorem~\ref{thm-extended-geo2}}.~  Let
$\{(\widetilde{X}_{\nu},\widetilde{J}_{\nu});\nu\in\bbZ_+\}$ denote a
block-monotone Markov chain with state space $\bbF$ and transition
probability matrix $\widetilde{\vc{P}}$. Conditions (i) and (ii) imply
that $\widetilde{\vc{P}}$ is irreducible and positive recurrent and
thus has the unique stationary probability vector, denoted by
$\widetilde{\vc{\pi}}=(\widetilde{\pi}(k,i))_{(k,i)\in\bbF}$.  It
follows from $\vc{P} \prec_d \widetilde{\vc{P}} \in \sfBM_d$ and
\cite[Proposition~2.3]{Masu15} that
$\vc{\pi}\prec_d\widetilde{\vc{\pi}}$ and thus
\begin{equation}
\sum_{k=0}^{\infty} \widetilde{\pi}(k,i)
= \sum_{k=0}^{\infty} \pi(k,i)
= \varpi(i),
\qquad i \in \bbD.
\label{eqn-tilde{J}=J}
\end{equation}

Following the last part of the proof of \cite[Theorem~4.1]{Masu15},
we have
\begin{equation}
\left\| \trunc{\vc{p}}_n^m(0,\vc{\varpi}) - \vc{p}^m(0,\vc{\varpi}) \right\|
\le 2m \sum_{i\in\bbD} \trunc{\widetilde{\pi}}_n(n,i),
\qquad m \in \bbN.
\label{add-eqn-17c}
\end{equation}
It follows from (\ref{add-eqn-17c}) and the triangle inequality
that, for $m \in \bbN$,
\begin{eqnarray}
\left\| \trunc{\vc{\pi}}_n - \vc{\pi} \right\|
&\le& \left\| \vc{p}^{mM}(0,\vc{\varpi}) - \vc{\pi} \right\|
+ \left\| \trunc{\vc{p}}_n^{mM}(0,\vc{\varpi}) - \trunc{\vc{\pi}}_n \right\|
\nonumber
\\
&& {} \quad 
+ \left\| \trunc{\vc{p}}_n^{mM}(0,\vc{\varpi}) - \vc{p}^{mM}(0,\vc{\varpi}) \right\|
\nonumber
\\
&\le& \left\| \vc{p}^{mM}(0,\vc{\varpi}) - \vc{\pi} \right\|
+ \left\| \trunc{\vc{p}}_n^{mM}(0,\vc{\varpi}) - \trunc{\vc{\pi}}_n \right\| 
+ 2mM \sum_{i\in\bbD} \trunc{\widetilde{\pi}}_n(n,i).
\label{bound-pi-02}
\end{eqnarray}

For $n \in \bbN$, let
$\trunc{\wt{\vc{P}}}_n:=(\trunc{\wt{p}}_n(k,i;\ell,j))_{(k,i),(\ell,j)\in\bbF}$
denote a stochastic matrix such that, for $i,j\in\bbD$,
\begin{equation*}
\trunc{\wt{p}}_n(k,i;\ell,j)
= 
\left\{
\begin{array}{ll}
\wt{p}(k,i;\ell,j), & k \in \bbZ_+,\ \ell = 0,1,\dots,n-1,
\\
\dm\sum_{m=n}^{\infty} \wt{p}(k,i;m,j), & k \in \bbZ_+,\ \ell = n,
\\
0, & \mbox{otherwise}.
\end{array}
\right.
\end{equation*}
Note that $\trunc{\wt{\vc{P}}}_n$ is the LC-block-augmented truncation
of $\wt{\vc{P}}$. Note also that $\trunc{\widetilde{\vc{P}}}_n \prec_d
\widetilde{\vc{P}} \in \sfBM_d$. Thus, it follows from
\cite[Proposition~2.3]{Masu15} that $(\trunc{\widetilde{\vc{P}}}_n)^M
\prec_d \widetilde{\vc{P}}^M$. Using this relation together with
$\vc{v} \in \sfBI_d$ and \cite[Remark~2.1]{Masu15}, we obtain
\begin{equation}
(\trunc{\widetilde{\vc{P}}}_n)^M\vc{v} \le \widetilde{\vc{P}}^M\vc{v}.
\label{ineqn-(n)wt{P}_nv}
\end{equation}
Combining (\ref{ineqn-hat{P}-widetilde{v}}) and
(\ref{ineqn-(n)wt{P}_nv}), we have
\[
(\trunc{\widetilde{\vc{P}}}_n)^M \vc{v}
\le \widetilde{\vc{P}}^M \vc{v}
\le \vc{v} - \phi \circ \vc{v}
+ b\vc{1}_K.
\]
Pre-multiplying this inequality by $\trunc{\widetilde{\vc{\pi}}}_n$
yields
\begin{equation}
\trunc{\widetilde{\pi}}_n(n,i) 
\le {b \over \phi \circ v(n,i)},\qquad i \in \bbD.
\label{add-eqn-0519-03}
\end{equation}
Substituting (\ref{add-eqn-0519-03}) into
(\ref{bound-pi-02}) results in
\begin{eqnarray}
\left\| \trunc{\vc{\pi}}_n - \vc{\pi} \right\|
&\le& \left\| \vc{p}^{mM}(0,\vc{\varpi}) - \vc{\pi} \right\|
+ \left\| \trunc{\vc{p}}_n^{mM}(0,\vc{\varpi}) - \trunc{\vc{\pi}}_n \right\|
\nonumber
\\
&& {} \quad 
+ 2mMb \sum_{i\in\bbD}{1 \over \phi \circ v(n,i)}.
\label{bound-pi-03}
\end{eqnarray}

We now define
$\{(\widetilde{X}_{\nu}^{(M)},\widetilde{J}_{\nu}^{(M)});\nu\in\bbZ_+\}$
as the $M$-skeleton of the Markov chain
$\{(\widetilde{X}_{\nu},\widetilde{J}_{\nu});\nu\in\bbZ_+\}$ with
transition probability matrix $\widetilde{\vc{P}}$, i.e.,
\[
\widetilde{X}_{\nu}^{(M)} = \widetilde{X}_{\nu M},\qquad 
\widetilde{J}_{\nu}^{(M)} = \widetilde{J}_{\nu M},
\qquad \nu \in \bbZ_+.
\]
Clearly, the $M$-skeleton
$\{(\widetilde{X}_{\nu}^{(M)},\widetilde{J}_{\nu}^{(M)})\}$ evolves
according to $\widetilde{\vc{P}}^M$. We also define
$\widetilde{\tau}_0^{(M)+}=\inf\{\nu\in\bbN;\widetilde{X}_{\nu}^{(M)}
= \widetilde{X}_{\nu M}=0\}$. Proceeding as in the derivation of the
first inequality for $\|\vc{p}^{m}(k,\vc{\varpi}) - \vc{\pi}
\|_{\svc{v}}$ at page 97 of \cite{Masu15} (replacing $\vc{v}$ by
$\vc{e}$ as in the derivation of (\ref{eqn-09})) and using
(\ref{eqn-tilde{J}=J}), we obtain
\begin{eqnarray}
\left\|\vc{p}^{mM}(k,\vc{\varpi}) - \vc{\pi}\right\|
&\le& 2\sum_{i\in\bbD}\varpi(i)
\wt{\PP}_{(k,i)}\big( \widetilde{\tau}_0^{(M)+} > m\big)
\nonumber
\\
&& {} 
 + 2\sum_{(\ell,j)\in\bbF}\widetilde{\pi}(\ell,j)
\wt{\PP}_{(\ell,j)}\big(\widetilde{\tau}_0^{(M)+} > m\big),
\qquad k \in \bbZ_+,\ m \in \bbN,
\label{eqn-21}
\end{eqnarray}
where $\wt{\PP}_{(k,i)}(\,\cdot\,) = \PP(\,\cdot\mid
\wt{X}_0=k,\wt{J}_0=i)$ for $(k,i) \in \bbF$. In addition, by Markov's
inequality,
\begin{equation}
\wt{\PP}_{(\ell,j)}\big(\wt{\tau}_0^{(M)+} > m\big) 
\le {1 \over r_{\phi}(m-1)}
\wt{\EE}_{(\ell,j)}\big[r_{\phi}\big(\wt{\tau}_0^{(M)+} - 1\big)\big],
\quad  m \in \bbN,\ (\ell,j)\in\bbF,
\label{ineqn-PP_{(l,j)}(tau_0^{(M)+}>m)}
\end{equation}
where
\begin{eqnarray*}
\wt{\EE}_{(\ell,j)}[\,\cdot\,] 
&=& \wt{\EE}[\,\cdot \mid \wt{X}_0 = \ell,\wt{J}_0=j],
\qquad (\ell,j) \in \bbF.
\end{eqnarray*}
Substituting (\ref{ineqn-PP_{(l,j)}(tau_0^{(M)+}>m)}) into
(\ref{eqn-21}), we have, for $m \in \bbN$,
\begin{eqnarray}
\lefteqn{
\left\|\vc{p}^{mM}(k,\vc{\varpi}) - \vc{\pi}\right\|
}
\quad &&
\nonumber
\\
&\le& {2 \over r_{\phi}(m-1)}
\left( 
\sum_{i\in\bbD}\varpi(i)
\wt{\EE}_{(k,i)}\big[r_{\phi}\big(\wt{\tau}_0^{(M)+} - 1 \big) \big]
+ 
\wt{\EE}_{\wt{\svc{\pi}}}\big[r_{\phi}\big(\wt{\tau}_0^{(M)+} - 1\big) \big]
\right),
\label{eqn-160303-01}
\end{eqnarray}
where
\begin{eqnarray*}
\wt{\EE}_{\wt{\svc{\pi}}}[\,\cdot\,] 
&=& \sum_{(k,i)\in\bbF}\wt{\pi}(k,i)\wt{\EE}_{(k,i)}[\,\cdot\,].
\end{eqnarray*}
It should be noted that if $K=0$ then Lemma~\ref{lem-bound-moment-tau}
yields
\begin{eqnarray}
\wt{\EE}_{(k,i)}\big[r_{\phi}\big(\wt{\tau}_0^{(M)+} - 1 \big) \big]
&\le& v(k \vee 1,i), \qquad (k,i) \in \bbF, 
\label{bound-moment-tau^{(M)+}-a-(K=0)}
\\
\wt{\EE}_{\wt{\svc{\pi}}}\big[r_{\phi}\big(\wt{\tau}_0^{(M)+} - 1\big) \big]
&\le& v(k \vee 1,\vc{\varpi}), \qquad  ~~\,k \in \bbZ_+.
\label{bound-moment-tau^{(M)+}-b-(K=0)}
\end{eqnarray}
Combining (\ref{eqn-160303-01}) with
(\ref{bound-moment-tau^{(M)+}-a-(K=0)}) and
(\ref{bound-moment-tau^{(M)+}-b-(K=0)}), we obtain
\begin{eqnarray}
\left\|  \vc{p}^{mM}(k,\vc{\varpi}) - \vc{\pi} \right\|
&\le& {4 \over r_{\phi}(m-1)} v(k \vee 1,\vc{\varpi}),
\qquad m \in \bbN.
\label{add-eqn-17a}
\end{eqnarray}
Similarly, we have
\begin{eqnarray}
\left\|
\trunc{\vc{p}}_n^{mM}(k,\vc{\varpi}) - \trunc{\vc{\pi}}_n  
\right\|
&\le& {4 \over r_{\phi}(m-1)} v(k \vee 1,\vc{\varpi}),
\qquad m,n \in \bbN.
\label{add-eqn-17b}
\end{eqnarray}
Substituting (\ref{add-eqn-17a}) and (\ref{add-eqn-17b}) into
(\ref{bound-pi-03}) yields the bound (\ref{bound-pi-K=0}) in the
special case where $K=0$.

In what follows, we consider the general case, where $K \in \bbN$. 
In fact, the following inequality holds (which is proved later):
\begin{equation}
\widetilde{\vc{P}}^M \widetilde{\vc{v}}
\le \widetilde{\vc{v}} - \widetilde{\phi} \circ \widetilde{\vc{v}} 
+ \widetilde{b}\vc{1}_0,
\label{ineqn-hat{P}v-02}
\end{equation}
where 
\begin{eqnarray}
\widetilde{\phi}(x)
&=& {\phi(1) \over \phi(B+1)}
\phi(x),\qquad x \ge 1,
\label{defn-phi}
\\
\widetilde{b} &=& b + B,
\label{defn-b}
\\
\widetilde{v}(k,i) &=& 
\left\{
\begin{array}{ll}
v(0,i), & k=0,~i\in\bbD,
\\
v(k,i) + B, & k\in\bbN,~i\in\bbD.
\end{array}
\right.
\label{defn-widetilde{v}}
\end{eqnarray}
Since $\widetilde{\vc{v}} \in \sfBI_d$ and
$(\trunc{\widetilde{\vc{P}}}_n)^M \prec_d \widetilde{\vc{P}}^M$, we
have $(\trunc{\widetilde{\vc{P}}}_n)^M \widetilde{\vc{v}} \le
\widetilde{\vc{P}}^M \widetilde{\vc{v}}$ (see
\cite[Remark~2.1]{Masu15}).  It follows from this inequality and
(\ref{ineqn-hat{P}v-02}) that
\begin{eqnarray}
(\trunc{\widetilde{\vc{P}}}_n)^M \widetilde{\vc{v}}
\le 
\widetilde{\vc{P}}^M \widetilde{\vc{v}}
\le \widetilde{\vc{v}} - \widetilde{\phi} \circ \widetilde{\vc{v}} 
+ \widetilde{b}\vc{1}_0.
\label{drift-cond-wt{P}^M}
\end{eqnarray}
Note here that $\widetilde{\phi}:[1,\infty) \to (0,\infty)$ is a
  nondecreasing differentiable concave function such that
  $\lim_{t\to\infty}\widetilde{\phi}'(t) = 0$. The inequality
  (\ref{drift-cond-wt{P}^M}) implies that the case where $K \in \bbN$
  is reduced to the special case where $K=0$.

To follow the proof of the case where $K=0$, we define two functions
$H_{\wt{\phi}}$ and $r_{\wt{\phi}}$, which correspond to the functions
$H_{\phi}$ and $r_{\phi}$, respectively.  Let $H_{\wt{\phi}}$ denote a
function on $[1,\infty)$ such that
\begin{equation}
H_{\wt{\phi}}(x) = \int_1^x {\rmd y \over \wt{\phi}(y)}
= \{ c_{\phi,B}(1) \}^{-1} H_{\phi}(x),
\qquad x \ge 1,
\label{defn-H_{wt{phi}}}
\end{equation}
where the second equality results from (\ref{defn-H_{phi}}),
(\ref{defn-c_{phi,B}}) and (\ref{defn-phi}).  Furthermore, let
$r_{\wt{\phi}}$ denote a function on $[0,\infty)$ such that
\[
r_{\wt{\phi}}(x) = \wt{\phi} \circ H_{\wt{\phi}}^{-1}(x),
\qquad x \ge 0.
\]
We then have
\begin{eqnarray}
r_{\wt{\phi}}(x)
&=&  c_{\phi,B}(1) \cdot \phi \circ H_{\wt{\phi}}^{-1}(x)
\nonumber
\\
&=& c_{\phi,B}(1) \cdot \phi \circ H_{\phi}^{-1}(c_{\phi,B}(x))
\nonumber
\\
&=& c_{\phi,B}(1) \cdot r_{\phi} \circ c_{\phi,B}(x),
\qquad x \ge 0,
\label{eqn-r_{wt{phi}}(x)}
\end{eqnarray}
where the first, second and third equalities follows from
(\ref{defn-phi}), (\ref{defn-H_{wt{phi}}}) and (\ref{defn-r_{phi}}),
respectively. In addition, following the derivation of
(\ref{add-eqn-17a}) and using $\wt{\vc{v}}$ in
(\ref{defn-widetilde{v}}) and $r_{\wt{\phi}}$ in
(\ref{eqn-r_{wt{phi}}(x)}) (instead of $\vc{v}$ and $r_{\phi}$), we
obtain
\begin{eqnarray}
\left\|  \vc{p}^{mM}(k,\vc{\varpi}) - \vc{\pi} \right\|
&\le& {4 \over r_{\wt{\phi}}(m-1)} \widetilde{v}(k \vee 1,\vc{\varpi})
\nonumber
\\
&=&
{4\{ c_{\phi,B}(1) \}^{-1} \over r_{\phi} \circ c_{\phi,B}(m-1)} 
\{ v(k \vee 1,\vc{\varpi}) + B\},\qquad k \in \bbZ_+.
\label{eqn-17a-02}
\end{eqnarray}
Similarly, 
\begin{eqnarray}
\qquad
\left\|
\trunc{\vc{p}}_n^{mM}(k,\vc{\varpi}) - \trunc{\vc{\pi}}_n  
\right\|
&\le&  
{4\{ c_{\phi,B}(1) \}^{-1} \over r_{\phi} \circ c_{\phi,B}(m-1)} 
\{ v(k \vee 1,\vc{\varpi}) + B\},\qquad k \in \bbZ_+.
\label{eqn-17b-02}
\end{eqnarray}
Substituting (\ref{eqn-17a-02}) and (\ref{eqn-17b-02}) into
(\ref{bound-pi-03}) yields (\ref{bound-pi-04}).

It remains to prove that (\ref{ineqn-hat{P}v-02}) holds. For
$k\in\bbZ_+$, let $\vc{v}(k)=(v(k,i))_{i\in\bbD}$ and
$\widetilde{\vc{v}}(k)=(\widetilde{v}(k,i))_{i\in\bbD}$, i.e.,
\[
\vc{v}^{\top} 
= (\vc{v}^{\top}(0),\vc{v}^{\top}(1),\vc{v}^{\top}(2),\dots),
\quad
\widetilde{\vc{v}}^{\top}
= (\widetilde{\vc{v}}^{\top}(0),\widetilde{\vc{v}}^{\top}(1),\widetilde{\vc{v}}^{\top}(2),\dots),
\]
where the superscript $``\top"$ denotes the transpose operator for
vectors and matrices.  It then follows from
(\ref{ineqn-hat{P}-widetilde{v}}), (\ref{defn-b}) and
(\ref{defn-widetilde{v}}) that
\begin{eqnarray}
\qquad
\sum_{\ell=0}^{\infty}
\widetilde{\vc{P}}^M(0;\ell)\widetilde{\vc{v}}(\ell)
&\le& \sum_{\ell=0}^{\infty}\widetilde{\vc{P}}^M(0;\ell)\vc{v}(\ell) + B\vc{e}
\le \vc{v}(0) - \phi \circ \vc{v}(0) 
+ (b +B)\vc{e}
\nonumber
\\
&=& \widetilde{\vc{v}}(0) - \phi \circ \widetilde{\vc{v}}(0)
+ \widetilde{b}\vc{e}
\le \widetilde{\vc{v}}(0) - \widetilde{\phi} \circ \widetilde{\vc{v}}(0) 
+ \widetilde{b}\vc{e},
\label{add-eqn-07}
\end{eqnarray}
where the last inequality holds because $\phi(x) \ge
\widetilde{\phi}(x)$ for $x \ge 1$ (see (\ref{defn-phi})).

It should be noted that $\widetilde{\vc{P}}^M \in \sfBM_d$ due to
$\widetilde{\vc{P}} \in \sfBM_d$. Thus, $\widetilde{\vc{P}}^M(k;0) \ge
\widetilde{\vc{P}}^M(K;0)$ for $k=0,1,\dots,K$. Using this and
(\ref{defn-widetilde{v}}), we have, for $k=1,2,\dots,K$,
\begin{eqnarray}
\sum_{\ell=0}^{\infty}
\widetilde{\vc{P}}^M(k;\ell)\widetilde{\vc{v}}(\ell)
&=& \sum_{\ell=0}^{\infty}\widetilde{\vc{P}}^M(k;\ell)\vc{v}(\ell)
+ B\{\vc{e} - \widetilde{\vc{P}}^M(k;0)\vc{e}\}
\nonumber
\\
&\le& \sum_{\ell=0}^{\infty}\widetilde{\vc{P}}^M(k;\ell)\vc{v}(\ell)
 + B\{\vc{e} - \widetilde{\vc{P}}^M(K;0)\vc{e}\}.
\label{add-eqn-10}
\end{eqnarray}
Applying (\ref{ineqn-hat{P}-widetilde{v}}) and
(\ref{defn-widetilde{v}}) to the right hand side of
(\ref{add-eqn-10}) yields, for $k=1,2,\dots,K$,
\begin{eqnarray}
\sum_{\ell=0}^{\infty}
\widetilde{\vc{P}}^M(k;\ell)\widetilde{\vc{v}}(\ell)
&\le& \vc{v}(k) - \phi \circ \vc{v}(k) 
+ b \vc{e} + B\{\vc{e} - \widetilde{\vc{P}}^M(K;0)\vc{e}\}
\nonumber
\\
&=& \widetilde{\vc{v}}(k) - \phi \circ \vc{v}(k) 
+ b \vc{e} - B\widetilde{\vc{P}}^M(K;0)\vc{e}.
\label{add-eqn-06}
\end{eqnarray}
Note here that the nonnegative concave function $\phi:[1,\infty) \to
  (0,\infty)$ is log-concave and thus
\begin{eqnarray*}
{\rmd \over \rmd x}\left( {\phi(x-B) \over \phi(x)} \right)
&=& {\phi(x-B) \over \phi(x)}
\left( {\phi'(x-B) \over \phi(x-B)} - {\phi'(x) \over \phi(x)} \right)
\nonumber
\\
&=& {\phi(x-B) \over \phi(x)}
\left\{ 
{\rmd \over \rmd x} \log \phi(x-B) - {\rmd \over \rmd x} \log \phi(x) 
\right\} \ge 0.
\end{eqnarray*}
Therefore, it follows from (\ref{defn-phi}) that
\begin{eqnarray*}
{ \phi(x-B) \over \phi(x) }
\ge {\phi(1) \over \phi(B+1)}
=
{ \widetilde{\phi}(x) \over \phi(x) },
\qquad x \ge B+1,
\end{eqnarray*}
which leads to 
\begin{equation}
\widetilde{\phi}(x) \le \phi(x-B),\qquad x \ge B+1.
\label{ineqn-wt{phi}(x)}
\end{equation}
In addition, from (\ref{defn-widetilde{v}}) and $\vc{v} \ge \vc{e}$,
we have
\begin{equation}
\widetilde{\vc{v}}(k) = \vc{v}(k) + B\vc{e} \ge (B+1)\vc{e},
\qquad k \in \bbN.
\label{ineqn-wt{v}(k)}
\end{equation}
The inequalities (\ref{ineqn-wt{phi}(x)}) and (\ref{ineqn-wt{v}(k)})
imply that
\begin{equation}
\widetilde{\phi} \circ \widetilde{\vc{v}}(k)
\le
\phi \circ \vc{v}(k), \qquad k \in \bbN,
\label{add-eqn-0519-04}
\end{equation}
where $\widetilde{\phi} \circ \widetilde{\vc{v}}(k) =
(\widetilde{\phi} \circ \widetilde{v}(k,i))_{i\in\bbD}$.  Substituting
(\ref{add-eqn-0519-04}) into (\ref{add-eqn-06}) and using
(\ref{defn-B}), we obtain
\begin{eqnarray}
\sum_{\ell=0}^{\infty}
\widetilde{\vc{P}}^M(k;\ell)\widetilde{\vc{v}}(\ell)
&\le& \widetilde{\vc{v}}(k) - \widetilde{\phi} \circ  \widetilde{\vc{v}}(k) 
+ 
\{  b \vc{e} - B \widetilde{\vc{P}}^M(K;0)\vc{e} \}
\nonumber
\\
&\le& \widetilde{\vc{v}}(k) - \widetilde{\phi} \circ  \widetilde{\vc{v}}(k), 
\qquad k=1,2,\dots,K.
\label{add-eqn-08}
\end{eqnarray}
Similarly, it follows from (\ref{ineqn-hat{P}-widetilde{v}}),
(\ref{defn-widetilde{v}}) and (\ref{add-eqn-0519-04}) that, for
$k=K+1,K+2,\dots$,
\begin{eqnarray}
\sum_{\ell=0}^{\infty}
\widetilde{\vc{P}}^M(k;\ell)\widetilde{\vc{v}}(\ell)
&\le& \sum_{\ell=0}^{\infty}\widetilde{\vc{P}}^M(k;\ell)\vc{v}(\ell) 
+ B\vc{e}
\le \vc{v}(k) - \phi \circ \vc{v}(k) 
+ B\vc{e}
\nonumber
\\
&=& \widetilde{\vc{v}}(k) - \phi \circ \vc{v}(k) 
\le \widetilde{\vc{v}}(k) - \widetilde{\phi} \circ \widetilde{\vc{v}}(k).
\label{add-eqn-09}
\end{eqnarray}
As a result, combining (\ref{add-eqn-07}), (\ref{add-eqn-08}) and
(\ref{add-eqn-09}) results in (\ref{ineqn-hat{P}v-02}). \qed

\section{Applications to GI/G/1-type Markov chains}\label{sec-applications}

In this section, we discuss the application of our results to
block-monotone GI/G/1-type Markov chains. We first present a procedure
for establishing an error bound for the LC-block-augmented truncation
of a general block-monotone GI/G/1-type Markov chain. We then consider
such a simple special case that the parameters of the error bound are
specified.

\subsection{General case}\label{subsec-general}

In this subsection, we consider a general block-monotone GI/G/1-type
Markov chain, which satisfies the following assumption.

\medskip

\begin{assumpt}\label{assumpt-GI/G/1-type}
\hfill
\begin{enumerate}
\renewcommand{\labelenumi}{(\roman{enumi})}
\item $\vc{P} \in \sfBM_d$ is an irreducible GI/G/1-type transition
  probability matrix such that
\begin{equation}
\vc{P}
=
\left(
\begin{array}{ccccc}
\vc{B}(0)  & \vc{B}(1)  & \vc{B}(2)  & \vc{B}(3) & \cdots
\\
\underline{\vc{A}}(-1) & \vc{A}(0)  & \vc{A}(1)  & \vc{A}(2) & \cdots
\\
\underline{\vc{A}}(-2) & \vc{A}(-1) & \vc{A}(0)  & \vc{A}(1) & \cdots
\\
\underline{\vc{A}}(-3) & \vc{A}(-2) & \vc{A}(-1) & \vc{A}(0) & \cdots
\\
\vdots     & \vdots     &  \vdots    & \vdots    & \ddots
\end{array}
\right),
\label{GIG1-type-P}
\end{equation}
where $\underline{\vc{A}}(-k) = \sum_{\ell=-\infty}^{-k}\vc{A}(\ell)$
for $k\in\bbZ_+$ and the diagonal blocks $\vc{A}(0)$ and $\vc{B}(0)$
are $d \times d$ matrices; and
\item $\vc{A}:=\sum_{k=-\infty}^{\infty}\vc{A}(k)$ is an irreducible
stochastic matrix.
\end{enumerate}
\end{assumpt}

\medskip

\begin{rem}\label{rem-Masu15}
The GI/G/1-type transition probability matrix $\vc{P}$ in
(\ref{GIG1-type-P}) above is equal to $\vc{P}$ in (5.1) of
\cite{Masu15}, though the latter has the block matrices
$\{\vc{B}(0),\vc{B}(-1),\vc{B}(-2),\dots\}$ in the first block
column. Indeed, for all $k\in\bbN$, $\vc{B}(-k)$ must be equal to
$\underline{\vc{A}}(-k)$ because $\vc{A} = \vc{B}(-k) +
\sum_{\ell=-k+1}^{\infty}\vc{A}(\ell)$ due to $\vc{P} \in \sfBM_d$
(see \cite[Proposition 1.1]{Masu15}).
\end{rem}

\medskip

It follows from (\ref{GIG1-type-P}), $\vc{P} \in \sfBM_d$ and
\cite[Proposition~1.1]{Masu15} that, for all $k\in\bbZ_+$,
\[
\vc{\Psi}
=\sum_{m=0}^{\infty}\vc{P}(k;m)
= \sum_{\ell=0}^{\infty}\vc{B}(\ell)
= \sum_{\ell=-\infty}^{\infty}\vc{A}(\ell)
= \vc{A}.
\]
Recall here that $\vc{\varpi}$ is the stationary probability vector of
$\vc{\Psi}$ and thus $\vc{A}$. Using $\vc{\varpi}$, we define $\sigma$
as
\begin{equation}
\sigma = \vc{\varpi}\sum_{k=-\infty}^{\infty}k\vc{A}(k)\vc{e}.
\label{defn-sigma}
\end{equation}
We then assume the following.

\medskip

\begin{assumpt}\label{assumpt-stability-cond}
$\sigma < 0$.
\end{assumpt}

\medskip

Under Assumptions~\ref{assumpt-GI/G/1-type} and
\ref{assumpt-stability-cond}, the GI/G/1-type transition probability
matrix $\vc{P}$ in (\ref{GIG1-type-P}) is irreducible and positive
recurrent (see, e.g., \cite[Chapter XI, Proposition 3.1]{Asmu03}) and
thus has the unique stationary probability vector
$\vc{\pi}=(\pi(k,i))_{(k,i)\in\bbF}$.

The subject of this subsection is to show how to obtain an upper bound
for $\| \trunc{\vc{\pi}}_n - \vc{\pi} \|$ by using
Theorem~\ref{thm-extended-geo2}. To apply
Theorem~\ref{thm-extended-geo2} to $\vc{P}$ in (\ref{GIG1-type-P}), we
have to prove that the drift condition
(\ref{ineqn-hat{P}-widetilde{v}}) holds with $\wt{\vc{P}}$ being equal
to $\vc{P}$, that is, $\vc{P}^M\vc{v} \le \vc{v} - \phi \circ \vc{v} +
b\vc{1}_K$. Unfortunately, since it is not easy to establish directly
the drift condition on $\vc{P}^M$, we construct a modified transition
probability matrix from $\vc{P}$ and establish a similar drift
condition on the $M$th power of the modified transition probability
matrix. Using the similar drift condition, we derive the desired drift
condition on $\vc{P}^M$.

We define the modified transition probability matrix mentioned above.
For $N \in \bbN$, let
$\vc{P}_N:=(p_N(k,i;\ell,j))_{(k,i),(\ell,j)\in\bbF}$ denote a
stochastic matrix such that
\begin{equation}
\vc{P}_N
=
\left(
\begin{array}{ccccc}
  \vc{B}(0)  
& \vc{B}(1)  
& \vc{B}(2)  
& \vc{B}(3)  
& \cdots
\\
  \vc{A}_N(-1) 
& \vc{A}_N(0)  
& \vc{A}_N(1)  
& \vc{A}_N(2)  
& \cdots
\\
   \vdots 
& \vdots 
& \vdots  
& \vdots 
& \ddots 
\\
  \vc{A}_N(-N) 
& \vc{A}_N(-N+1) 
& \vc{A}_N(-N+2) 
& \vc{A}_N(-N+3) 
& \cdots
\\
 \vc{O}
& \vc{A}_N(-N) 
& \vc{A}_N(-N+1) 
& \vc{A}_N(-N+2) 
& \cdots 
\\
   \vc{O}
&  \vc{O}
&  \vc{A}_N(-N) 
& \vc{A}_N(-N+1) 
& \cdots
\\
  \vdots    
& \vdots     
& \vdots    
& \vdots    
& \ddots
\end{array}
\right),
\label{defn-P_N}
\end{equation}
where $\vc{A}_N(k)$, $k \in \bbZ:=\{0,\pm1,\pm2,\dots\}$ is given by
\begin{equation}
\vc{A}_N(k)
= \left\{
\begin{array}{ll}
\vc{O}, & k \le -N-1,
\\
\underline{\vc{A}}(-N), & k=-N,
\\
\vc{A}(k), & k \ge -N+1.
\end{array}
\right.
\label{defn-widetilde{A}(k)}
\end{equation}
It follows from (\ref{defn-P_N}) and
Assumption~\ref{assumpt-GI/G/1-type} that $\vc{P} \prec_d \vc{P}_N \in
\sfBM_d$ for $N \in \bbN$.  Furthermore, let
\begin{eqnarray}
\sigma_N
=
\vc{\varpi}
\sum_{k=-\infty}^{\infty}k\vc{A}_N(k)\vc{e}.
\label{defn-sigma_N-02}
\end{eqnarray}
Assumption \ref{assumpt-stability-cond} implies that, for
all sufficiently large $N \in \bbN$,
\begin{equation}
\sigma_N  < 0.
\label{cond-sigma_N<0}
\end{equation}

In the rest of this subsection, we fix $N \in \bbN$ such that
(\ref{cond-sigma_N<0}) holds. For $M \in \bbN$, we then define
$\{\vc{A}_N^{\ast M}(k);k\in\bbZ\}$ as the $M$th-fold convolution of
$\{\vc{A}_N(k);k\in\bbZ\}$, i.e., $\vc{A}_N^{\ast 1}(k) = \vc{A}_N(k)$
for $k \in \bbZ$ and, for $M \ge 2$,
\[
\vc{A}_N^{\ast M}(k) 
= \sum_{\ell \in \bbZ}\vc{A}_N^{\ast (M-1)}(k-\ell) \vc{A}_N(\ell),
\qquad k \in \bbZ.
\]
From (\ref{defn-widetilde{A}(k)}), we have $\vc{A}_N^{\ast
  M}(k)=\vc{O}$ for all $k \le -L - 1$, where $L = MN$. We also obtain
the following result.

\medskip

\begin{lemma}\label{lem-M_0}
There exists some $M_0 \in \bbN$ such that
\begin{equation}
\sum_{\ell=-L}^{\infty} \ell \vc{A}_N^{\ast M}(\ell)\vc{e} 
< \vc{0}\qquad \mbox{for all $M \ge M_0$}.
\label{defn-sigma_N-04}
\end{equation}
\end{lemma}

\medskip

\begin{proof}
We consider a Markov additive process
$\{(Y_{\nu},J_{\nu});\nu\in\bbZ_+\}$ with state space $\bbZ \times
\bbD$ and kernel $\{\vc{A}_N(k);k\in\bbZ\}$, which evolves according
to the following transition law.
\[
\PP(Y_{\nu+1} - Y_{\nu} = k, J_{\nu+1} = j \mid J_{\nu} = i)
= [\vc{A}_N(k)]_{i,j},
\qquad k \in \bbZ,\ i,j \in \bbD,
\]
where $[\,\cdot\,]_{i,j}$ denotes the $(i,j)$th element of the matrix
in the square brackets. Note here that $\vc{\varpi}$ is the stationary
probability vector of the background Markov chain
$\{J_{\nu};\nu\in\bbZ_+\}$ and thus $\sigma_N < 0$ is the mean drift
of the marginal process $\{Y_{\nu};\nu\in\bbZ_+\}$. In addition, it
follows from $\sigma_N < 0$ that $\PP(\lim_{\nu\to\infty}Y_{\nu} =
-\infty \mid Y_0 = i) = 1$ for all $i \in \bbD$ (see, e.g.,
\cite[Chapter XI, Proposition 2.10]{Asmu03}). Therefore,
(\ref{defn-sigma_N-04}) holds for some $M_0 \in \bbN$.
\end{proof}

\medskip

To proceed, we fix $M \ge M_0$ arbitrarily, where $M_0$ is some
positive integer satisfying (\ref{defn-sigma_N-04}). We denote by
$\vc{P}_N^M(k;\ell)$, $k,\ell \in \bbZ_+$ the $(k,\ell)$th block of
the $M$th power $\vc{P}_N^M$ of $\vc{P}_N$. It then follows from
(\ref{defn-P_N}) that
\[
\vc{P}_N^M(k;\ell)
= 
\left\{
\begin{array}{ll}
\vc{A}_N^{\ast M}(l-k), & k \ge L,\ \ell \ge k - L,
\\
\vc{O},                           & k \ge L,\ \ell \le k - L - 1,
\end{array}
\right.
\]
and thus
\begin{equation}
\vc{P}_N^M
=
\left(
\begin{array}{ccccc}
  \vc{P}_N^M(0;0)  
& \vc{P}_N^M(0;1)
& \vc{P}_N^M(0;2)  
& \vc{P}_N^M(0;3)  
& \cdots
\\
  \vc{P}_N^M(1;0)  
& \vc{P}_N^M(1;1)  
& \vc{P}_N^M(1;2) 
& \vc{P}_N^M(1;3)  
& \cdots
\\
   \vdots 
& \vdots 
& \vdots  
& \vdots  
& \ddots 
\\
  \vc{P}_N^M(L-1;0)  
& \vc{P}_N^M(L-1;1)  
& \vc{P}_N^M(L-1;2)  
& \vc{P}_N^M(L-1;3)  
& \cdots
\\
  \vc{A}_N^{\ast M}(-L) 
& \vc{A}_N^{\ast M}(-L+1) 
& \vc{A}_N^{\ast M}(-L+2) 
& \vc{A}_N^{\ast M}(-L+3)
& \cdots
\\
 \vc{O}
&  \vc{A}_N^{\ast M}(-L)
& \vc{A}_N^{\ast M}(-L+1) 
& \vc{A}_N^{\ast M}(-L+2) 
& \cdots 
\\
   \vc{O}
&  \vc{O}
&  \vc{A}_N^{\ast M}(-L) 
& \vc{A}_N^{\ast M}(-L+1)
& \cdots
\\
  \vdots    
& \vdots  
& \vdots      
& \vdots       
& \ddots
\end{array}
\right).
\label{defn-P_N^M}
\end{equation}
Since $\vc{P} \prec_d \vc{P}_N \in \sfBM_d$, it follows from
\cite[Remark~2.1 and Proposition~2.3]{Masu15} that $\vc{P}^{M} \prec_d
\vc{P}_N^M$, which implies that
\begin{equation}
\vc{P}^{M}\vc{v} \le \vc{P}_N^M \vc{v}
\quad \mbox{if $\vc{v} \in \sfBI_d$}.
\label{ineqn-P^Mv}
\end{equation}
Therefore, we show that $\vc{P}_N^M$ satisfies a drift condition
similar to (\ref{ineqn-hat{P}-widetilde{v}}).

We now introduce a function $V: [0,\infty) \to [1,\infty)$, which
    plays a key role in establishing the drift condition on
    $\vc{P}_N^M$.
\medskip

\begin{assumpt}\label{assumpt-moment-cond}
There exists some $\alpha \in [0,1)$ such that, for any $\delta > 0$,
\begin{eqnarray}
\lim_{k\to\infty}
{1 \over V'(k)}
\sum_{\ell = \lfloor \delta k^{1 - \alpha} \rfloor + 1}^{\infty}
V(k + \ell)\vc{A}_N^{\ast M}(\ell)\vc{e} 
&=& \vc{0},
\label{lim-V'(k)}
\end{eqnarray}
where $V: [0,\infty) \to [1,\infty)$ is an increasing, convex and
    log-concave function such that
\begin{enumerate}
\renewcommand{\labelenumi}{(\roman{enumi})}
\item $\lim_{x\to\infty}(\log V(x)) / x = \log r_{A_+}$,
where
\[
r_{A_+} = \sup\left\{z \ge 1; 
\mbox{$\sum_{\ell = 0}^{\infty}z^{\ell}\vc{A}(l)$ is finite}\right\};
\]
\item $V(x)$ is twice-differentiable for $x > 0$;
\item $V'(0) := \lim_{x \downarrow 0}V'(x) > 0$ and $\lim_{x\to\infty}V'(x) = \infty$;
\item $V''(x)/V'(x)$ is nonincreasing for $x > 0$; and
\item $\limsup_{\delta \downarrow 0}\lim_{x\to\infty}
V'(x+\delta x^{1 - \alpha})/V'(x) = 1$.
\end{enumerate}
\end{assumpt}

\medskip

\begin{rem}
Appendix~\ref{appen-examp-F} provides a sufficient condition for
Assumption~\ref{assumpt-moment-cond} and the typical examples of $V$.
\end{rem}

\medskip

The key properties of the function $V$ are summarized in
Lemmas~\ref{lem-condition(i)} and \ref{lem-K_0} below.

\medskip

\begin{lemma}\label{lem-condition(i)}
Suppose that Assumption~\ref{assumpt-moment-cond} is satisfied. It
then holds that, for any $\alpha_0 \in (\alpha,1)$,
\begin{eqnarray}
V'(x) = \exp\{o(x^{\alpha_0})\},
\label{eqn-V'(x)}
\end{eqnarray}
where $f(x) = o(g(x))$ represents $\lim_{x\to\infty}f(x)/g(x) =
0$. It also holds that
\begin{equation}
\log V(x) = o(x^{\alpha_0}).
\label{eqn-logV(x)/x}
\end{equation}
In addition, 
\begin{equation}
\lim_{x\to\infty}{V''(x) \over V'(x)} = 0.
\label{lim-V''(x)/V'(x)}
\end{equation}

\end{lemma}

\medskip

\begin{proof}
We first prove (\ref{eqn-V'(x)}).
Let $G$ denote a continuous function
on $[0,\infty)$ such that
\begin{equation*}
G(x) = 1 - \dm{V'(0) \over V'(x)}, \qquad x \ge 0.
\end{equation*}
Since $V'$ is nondecreasing and $\lim_{x\to\infty}V'(x) = \infty$ (due
to the convexity of $V$ and condition (iii) of
Assumption~\ref{assumpt-moment-cond}), $G$ is a distribution
function. We now define $\overline{G}(x) = 1 - G(x)$ for $x \ge 0$,
i.e.,
\begin{equation}
\overline{G}(x) = {V'(0) \over V'(x)}, \qquad x \ge 0.
\label{def-ol{G}}
\end{equation}
It then follows from condition (v) of
Assumption~\ref{assumpt-moment-cond} that
\[
\liminf_{\delta \downarrow 0}\lim_{x\to\infty}
{\overline{G}(x+\delta x^{1 - \alpha}) \over \overline{G}(x)} = 1,
\]
and thus, for any $\alpha_0 \in (\alpha,1)$,
\begin{eqnarray*}
1 \ge 
\lim_{x\to\infty} 
{\overline{G}(x+x^{1 - \alpha_0}) \over \overline{G}(x)}
\ge 
\liminf_{\delta \downarrow 0}
\lim_{x\to\infty} {\overline{G}(x+\delta x^{1 - \alpha}) \over \overline{G}(x)}
= 1.
\end{eqnarray*}
Therefore, we have
\[
\lim_{x\to\infty} 
{\overline{G}(x+x^{1 - \alpha_0}) \over \overline{G}(x)} = 1,
\]
which implies that the distribution function $G$ is
$(1/\alpha_0)$th-order long-tailed \cite[Definition~1.1 and Lemma
  A.3]{Masu13}. Combining this fact with \cite[Lemma A.1~(i)]{Masu13}
yields
\begin{equation}
\overline{G}(x) = \exp\{-o(x^{\alpha_0})\}.
\label{eqn-ol{G}}
\end{equation}
Substituting (\ref{eqn-ol{G}}) into (\ref{def-ol{G}}) results in
(\ref{eqn-V'(x)}). Furthermore, using (\ref{eqn-V'(x)}), we can
readily obtain (\ref{eqn-logV(x)/x}).

Finally, we prove (\ref{lim-V''(x)/V'(x)}) by contradiction. To this
end, we assume that (\ref{lim-V''(x)/V'(x)}) does not hold, i.e.,
there exist some $\delta > 0$ and $x_0 := x_0(\delta)> 0$ such that
\[
{\rmd \over \rmd x }\log V'(x)
= {V''(x) \over V'(x)} \ge \delta \qquad \mbox{for all $x > x_0$},
\] 
and thus $\log V'(x) - \log V'(x_0) \ge \delta (x-x_0)$ for $x \ge
x_0$, which yields
\[
{V'(x) \over V'(x_0)} \ge  \exp\{\delta (x - x_0)\},
\qquad x \ge x_0.
\]
Therefore, we have
\[
V(x) \ge V(x_0) + {V'(x_0) \over \delta}(\rme^{\delta(x-x_0)} - 1),
\qquad x \ge x_0.
\]
From this inequality, we obtain $\liminf_{x\to\infty}\log V(x) / x \ge
\delta$, which is inconsistent with (\ref{eqn-logV(x)/x}).
\end{proof}

\medskip

\begin{rem}\label{rem-r_{A_+}}
Equation (\ref{eqn-logV(x)/x}) shows that $\log V(x) = o(x)$. Thus,
from condition (i) of Assumption~\ref{assumpt-moment-cond}, we have
$r_{A_+} = 1$, i.e., $\{\vc{A}(k);k\in\bbZ_+\}$ is heavy-tailed.
\end{rem}

\medskip

\begin{rem}
As mentioned in Remark~\ref{rem-Masu15}, Masuyama~\cite{Masu15}
considered the same block-monotone GI/G/1-type transition probability
matrix as $\vc{P}$ in (\ref{GIG1-type-P}). However,
Masuyama~\cite{Masu15} assumed that $r_{A_+} > 1$, i.e.,
$\{\vc{A}(k);k\in\bbZ_+\}$ is light-tailed, which implies that
$\{\vc{B}(k);k\in\bbZ_+\}$ is also light-tailed. Indeed, since $\vc{P}
\in \sfBM_d$, it holds that $\sum_{\ell=k}^{\infty}\vc{B}(\ell) \le
\sum_{\ell=k}^{\infty}\vc{A}(\ell)$ for all $k \in \bbZ_+$, which
yields, for $z \in [1,r_{A_+})$,
\begin{eqnarray*}
\sum_{k=0}^{\infty}z^k \vc{B}(k)
&=& \sum_{\ell=0}^{\infty} \vc{B}(\ell) 
+ \sum_{\ell=1}^{\infty} (z^{\ell} - 1) \vc{B}(\ell)  
\nonumber
\\
&=& \sum_{\ell=0}^{\infty} \vc{B}(\ell) 
  + \sum_{k=1}^{\infty}(z^k - z^{k-1}) \sum_{\ell=k}^{\infty} \vc{B}(\ell) 
\nonumber
\\
&\le& \sum_{\ell=0}^{\infty} \vc{A}(\ell) 
  + \sum_{k=1}^{\infty}(z^k - z^{k-1}) \sum_{\ell=k}^{\infty} \vc{A}(\ell) 
= \sum_{k=0}^{\infty}z^k \vc{A}(k).
\end{eqnarray*}
Thus, if $r_{A_+} > 1$, then $r_{B_+}:=\sup\left\{z \ge 1;
\mbox{$\sum_{\ell = 0}^{\infty}z^{\ell}\vc{B}(l)$ is finite}\right\}
\ge r_{A_+} > 1$. It is known (see \cite[Theorem 2.2]{Jarn03} and
\cite[Theorem 3.1]{Mao14}) that, under
Assumptions~\ref{assumpt-GI/G/1-type} and
\ref{assumpt-stability-cond}, $\min(r_{A_+}, r_{B_+}) > 1$ if and only
if the GI/G/1-type transition probability matrix $\vc{P}$ in
(\ref{GIG1-type-P}) satisfies the geometric drift condition (see
Remark~\ref{rem-drift-cond}).  On the other hand,
Assumption~\ref{assumpt-moment-cond} together with
Assumptions~\ref{assumpt-GI/G/1-type} and \ref{assumpt-stability-cond}
implies that the GI/G/1-type transition probability matrix $\vc{P}$ in
(\ref{GIG1-type-P}) is positive recurrent but does not satisfy the
geometric drift condition because $r_{A_+} = 1$, i.e.,
$\{\vc{A}(k);k\in\bbZ_+\}$ is heavy-tailed (see
Remark~\ref{rem-r_{A_+}}). It is also known (see
\cite[Theorem~3.1]{Li05}) that if $\{\vc{A}(k);k\in\bbZ_+\}$ is
heavy-tailed then so is the stationary probability vector of the
ergodic GI/G/1-type transition probability matrix $\vc{P}$. Such
GI/G/1-type transition probability matrices typically arise from
BMAP/GI/1 queues with subexponential service times and/or batch sizes,
and these Markov chains have the subexponential stationary probability
vectors under some mild technical conditions \cite{Masu09,Masu15-ANOR}
(see also \cite{Kimu13,Masu11}).
\end{rem}

\medskip

\begin{lemma}\label{lem-K_0}
If Assumption~\ref{assumpt-moment-cond} holds, then, for any
$\varepsilon \in (0,1)$, there exist some
$\delta_0:=\delta_0(\varepsilon) > 0$ and positive integer
$K_0:=K_0(\delta_0,L) \ge L$ such that
\begin{eqnarray}
V'(k + \delta_0 k^{1 - \alpha}) 
&\le& (1 + \varepsilon)V'(k),
\qquad  k \ge K_0+1,
\label{ineqn-K_0-01}
\\
V'(k - L) 
&\ge& (1 - \varepsilon)V'(k),
\qquad  k \ge K_0+1.
\label{ineqn-K_0-02}
\end{eqnarray}
\end{lemma}

\begin{proof}
Recall that $V'$ is nondecreasing. Thus, it follows from condition (v)
of Assumption~\ref{assumpt-moment-cond} that
\begin{eqnarray*}
1 \ge \lim_{x\to\infty}
{V'(x - L) \over V'(x)}
&=&
\lim_{y\to\infty}
{V'(y) \over V'(y+L)}
\ge \liminf_{\delta \downarrow 0}\lim_{y\to\infty}
{V'(y) \over V'(y + \delta y^{1 - \alpha})}
= 1,
\end{eqnarray*}
which leads to
\begin{equation}
\lim_{x\to\infty}
{V'(x - L) \over V'(x)} = 1.
\label{long-tailed-V'}
\end{equation}
In addition, condition (v) of Assumption~\ref{assumpt-moment-cond}
implies that, for any $\varepsilon \in (0,1)$,
\begin{equation}
\lim_{x\to\infty}
{V'(x + \delta_0 x^{1 - \alpha}) \over V'(x)}
\le 1 + 2 \varepsilon,
\label{ineqn-lim-V'}
\end{equation}
where $\delta_0 > 0$ is sufficiently small depending on $\varepsilon
\in (0,1)$. Consequently, (\ref{long-tailed-V'}) and
(\ref{ineqn-lim-V'}) show that the statement of this lemma is true.
\end{proof}

\medskip

To establish the drift condition on $\vc{P}_N^M$, we estimate
$\sum_{\ell=0}^{\infty} \vc{P}_N^M(k;\ell)V(\ell)\vc{e}$ for
sufficiently large $k$'s by using Lemmas~\ref{lem-M_0} and \ref{lem-K_0}.
Lemma~\ref{lem-M_0} implies that there exist
$\kappa > 0$ and $\varepsilon \in (0,1)$ such that
\begin{equation}
(1-\varepsilon) 
\sum_{\ell=-L}^{\infty} 
\ell \vc{A}_N^{\ast M}(\ell)\vc{e}
+  2\varepsilon
\sum_{\ell = 0}^{\infty} 
\ell \vc{A}_N^{\ast M}(\ell)\vc{e}
\le -2\kappa \vc{e}.
\label{ineqn-varepsilon}
\end{equation}
Furthermore, according to (\ref{lim-V'(k)}), there exists some $K \in
\{K_0, K_0+1,\dots\}$ such that
\begin{eqnarray}
{1 \over V'(k)}
\sum_{\ell = \lfloor \delta_0 k^{1 - \alpha} \rfloor + 1}^{\infty}
V(k + \ell)\vc{A}_N^{\ast M}(\ell)\vc{e}
&\le& \kappa \vc{e}\qquad \mbox{for all $k \ge K+1$},
\label{ineqn-K_1-02}
\end{eqnarray}
where $\delta_0 > 0$ and $K_0 \in \{L,L+1,\dots\}$ are fixed such that
(\ref{ineqn-K_0-01}) and (\ref{ineqn-K_0-02}) hold. From
(\ref{ineqn-varepsilon}), (\ref{ineqn-K_1-02}) and
Lemma~\ref{lem-K_0}, we obtain the following result.

\medskip

\begin{lemma}\label{lem-App-01}
Suppose that Assumptions~\ref{assumpt-GI/G/1-type},
\ref{assumpt-stability-cond} and \ref{assumpt-moment-cond} are
satisfied. Let $L=MN$, where $N\in \bbN$ and $M\in \bbN$ are fixed
such that (\ref{cond-sigma_N<0}) and (\ref{defn-sigma_N-04})
hold. Furthermore, fix
\begin{enumerate}
\item $\kappa >0$ and $\varepsilon \in (0,1)$ such that
  (\ref{ineqn-varepsilon}) holds;
\item $\delta_0 > 0$ and $K_0 \in \{L,L+1,\dots\}$ such that
  (\ref{ineqn-K_0-01}) and (\ref{ineqn-K_0-02}) hold; and
\item $K \in \{K_0, K_0+1,\dots\}$ such that (\ref{ineqn-K_1-02})
  holds.
\end{enumerate}
We then have
\begin{equation}
\sum_{\ell=0}^{\infty}
\vc{P}_N^M(k;\ell)V(\ell)\vc{e} - V(k)\vc{e}
\le - \kappa V'(k)\vc{e}, \qquad k \ge K+1.
\label{ineqn-Pv-v-02}
\end{equation}
\end{lemma}

\medskip

\begin{proof}
It follows from (\ref{defn-P_N^M}) and $K \ge K_0 \ge L$ that, for $k
\ge K+1$,
\begin{eqnarray}
\lefteqn{
\sum_{\ell=0}^{\infty}\vc{P}_N^M(k;\ell)V(\ell)\vc{e} - V(k)\vc{e}
}
\qquad &&
\nonumber
\\
&=& \sum_{\ell=-L}^{\infty}\vc{A}_N^{\ast M}(\ell)V(k + \ell)\vc{e} 
- V(k)\vc{e} 
\nonumber
\\
&=& 
\sum_{\ell = \lfloor \delta_0 k^{1 - \alpha} \rfloor + 1}^{\infty}
\vc{A}_N^{\ast M}(\ell)V(k + \ell)\vc{e}
+ \sum_{\ell=-L}^{\lfloor \delta_0 k^{1 - \alpha} \rfloor}
\vc{A}_N^{\ast M}(\ell)V(k + \ell)\vc{e}
- V(k)\vc{e}, \qquad
\label{eqn-Pv-v-01}
\end{eqnarray}
where, by convention, any empty sum (which has no terms) is defined as
zero. It also follows from the mean value theorem that, for any $k \in
\bbZ_+$ and $-L \le \ell \le \delta_0 k^{1 - \alpha}$, there exists
some $\xi \in (0,1)$ such that
\begin{equation}
V(k + \ell) = V(k) + \ell V'(k + \xi \ell), 
\qquad k \in \bbZ_+,\, -L \le \ell \le \delta_0 k^{1 - \alpha}.
\label{ineqn-01}
\end{equation}
Substituting (\ref{ineqn-01}) into the second term in the right hand
side of (\ref{eqn-Pv-v-01}) yields, for $k \ge K+1$,
\begin{eqnarray}
\lefteqn{
\sum_{\ell=0}^{\infty}\vc{P}_N^M(k;\ell)V(\ell)\vc{e} 
- V(k)\vc{e}
}
\qquad &&
\nonumber
\\
&=& \sum_{\ell = \lfloor \delta_0 k^{1 - \alpha} \rfloor + 1}^{\infty}
V(k + \ell)\vc{A}_N^{\ast M}(\ell)\vc{e}
+ \sum_{\ell=-L}^{\lfloor \delta_0 k^{1 - \alpha} \rfloor}
V'(k + \xi \ell) \ell \vc{A}_N^{\ast M}(\ell)\vc{e}
\nonumber
\\
&& {}
+ V(k)\sum_{\ell=-L}^{\lfloor \delta_0 k^{1 - \alpha} \rfloor}
\vc{A}_N^{\ast M}(\ell)\vc{e}
 - V(k)\vc{e}
\nonumber
\\
&\le& \sum_{\ell = \lfloor \delta_0 k^{1 - \alpha} \rfloor + 1}^{\infty}
V(k + \ell)\vc{A}_N^{\ast M}(\ell)\vc{e}
+ \sum_{\ell=-L}^{\lfloor \delta_0 k^{1 - \alpha} \rfloor}
V'(k + \xi \ell) \ell \vc{A}_N^{\ast M}(\ell)\vc{e}, 
\label{eqn-Pv-v-02}
\end{eqnarray}
where the inequality holds because $\sum_{\ell=-L}^{\lfloor \delta_0
  k^{1 - \alpha} \rfloor}\vc{A}_N^{\ast M}(\ell)\vc{e} \le \vc{e}$ for
all $k \in \bbZ_+$.  Furthermore, since $V'$ is nondecreasing, we have
\begin{eqnarray*}
\sum_{\ell = -L}^{\lfloor \delta_0 k^{1 - \alpha} \rfloor} 
V'(k + \xi \ell) \ell \vc{A}_N^{\ast M}(\ell)\vc{e}
&\le& V'(k - L) 
\sum_{\ell = -L}^{-1} \ell \vc{A}_N^{\ast M}(\ell)\vc{e}
\nonumber
\\
&& {} 
+  V'(k + \delta_0 k^{1 - \alpha})
\sum_{\ell = 0}^{\lfloor \delta_0 k^{1 - \alpha} \rfloor} 
\ell \vc{A}_N^{\ast M}(\ell)\vc{e},
\quad k \ge K+1.
\end{eqnarray*}
Applying (\ref{ineqn-K_0-01}) and (\ref{ineqn-K_0-02}) to the above
inequality, we obtain, for $k \ge K+1$,
\begin{eqnarray}
\lefteqn{
\sum_{\ell = -L}^{\lfloor \delta_0 k^{1 - \alpha} \rfloor}
V'(k + \xi \ell) \ell\vc{A}_N^{\ast M}(\ell)\vc{e}
}
\quad &&
\nonumber
\\
&\le& V'(k) 
\left[
(1-\varepsilon)
\sum_{\ell = -L}^{-1} \ell \vc{A}_N^{\ast M}(\ell)\vc{e}
+  (1+\varepsilon)
\sum_{\ell = 0}^{\lfloor \delta_0 k^{1 - \alpha} \rfloor} 
\ell \vc{A}_N^{\ast M}(\ell)\vc{e}
\right]
\nonumber
\\
&=& V'(k)
\left[
(1-\varepsilon) 
\sum_{\ell=-L}^{\lfloor \delta_0 k^{1 - \alpha} \rfloor} 
\ell \vc{A}_N^{\ast M}(\ell)\vc{e}
+  2\varepsilon
\sum_{\ell = 0}^{\lfloor \delta_0 k^{1 - \alpha} \rfloor} 
\ell \vc{A}_N^{\ast M}(\ell)\vc{e}
\right]
\nonumber
\\
&\le& V'(k)
\left[
(1-\varepsilon) 
\sum_{\ell=-L}^{\infty} 
\ell \vc{A}_N^{\ast M}(\ell)\vc{e}
+  2\varepsilon
\sum_{\ell = 0}^{\infty} 
\ell \vc{A}_N^{\ast M}(\ell)\vc{e}
\right]
\nonumber
\\
&\le& -2\kappa V'(k)\vc{e},
\label{add-eqn-04}
\end{eqnarray}
where we use (\ref{ineqn-varepsilon}) in the last inequality.
Finally, substituting (\ref{ineqn-K_1-02}) and (\ref{add-eqn-04}) into
(\ref{eqn-Pv-v-02}) yields (\ref{ineqn-Pv-v-02}).
\end{proof}

\medskip

We are now ready to establish the drift condition on $\vc{P}_N^M$ and
thus $\vc{P}^M$.  We fix
\begin{align}
&&&&&&&& \vc{v}(k) &= V(k)\vc{e}, & k &\in \bbZ_+, &&&&&&&&
\label{fix-v(k)}
\\
&&&&&&&& \phi(t) &= \kappa V'(V^{-1}(t)), & t &\ge 1, &&&&&&&&
\label{defn-tilde{phi}}
\end{align}
where $V^{-1}$ denotes the inverse function of the increasing and
differentiable function $V$ (see
Assumption~\ref{assumpt-moment-cond}). Clearly, $\vc{v} \in
\sfBM_d$. Furthermore, it follows from (\ref{defn-tilde{phi}}) that
\begin{equation}
\kappa V'(k)\vc{e} = \phi \circ V(k)\vc{e}
= \phi \circ \vc{v}(k),\qquad k \in \bbZ_+.
\label{eqn-kappa*V'(t)}
\end{equation}
Substituting (\ref{fix-v(k)}) and (\ref{eqn-kappa*V'(t)}) into
(\ref{ineqn-Pv-v-02}), we have
\begin{equation}
\sum_{\ell=0}^{\infty}
\vc{P}_N^M(k;\ell) \vc{v}(\ell) - \vc{v}(k)
\le - \phi \circ \vc{v}(k), \qquad k \ge K+1.
\label{ineqn-Pv-v-06}
\end{equation}
We also have
\begin{equation}
\sum_{\ell=0}^{\infty}
\vc{P}_N^M(k;\ell)\vc{v}(\ell) - \vc{v}(k)
\le - \phi \circ \vc{v}(k)
+  b \vc{e}, 
\qquad k=0,1,\dots,K,
\label{ineqn-Pv-05}
\end{equation}
where
\begin{equation}
b  = \inf\left\{
x \ge 0; x\vc{e} \ge \sum_{\ell=0}^{\infty}\vc{P}_N^M(k;\ell)
\vc{v}(\ell) - \vc{v}(k) + \phi 
\circ \vc{v}(k), k=0,1,\dots,K
\right\}.
\label{defn-widetilde{b}}
\end{equation}
Combining (\ref{ineqn-Pv-v-06}) and (\ref{ineqn-Pv-05}) yields
\begin{equation}
\vc{P}_N^M\vc{v}
\le \vc{v} - \phi \circ \vc{v} +  b \vc{1}_K.
\label{ineqn-P_N_widetilde{v}}
\end{equation}
From $\vc{v} \in \sfBM_d$, (\ref{ineqn-P^Mv}) and
(\ref{ineqn-P_N_widetilde{v}}), we obtain
\begin{equation}
\vc{P}^M\vc{v} 
\le \vc{v} - \phi \circ \vc{v}
+  b \vc{1}_K.
\label{ineqn-P_widetilde{v}}
\end{equation}

To use Theorem~\ref{thm-extended-geo2}, it remains to show that the
function $\phi$ in (\ref{defn-tilde{phi}}) is qualified for the
function $\phi$ appearing in Theorem~\ref{thm-extended-geo2}, which is
accomplished by Lemma~\ref{lem-phi} below.

\medskip

\begin{lemma}\label{lem-phi}
The function $\phi$ in (\ref{defn-tilde{phi}}) is
nondecreasing, differentiable and concave. Furthermore,
$\lim_{t\to\infty}\phi'(t) = 0$.
\end{lemma}

\medskip

{\it Proof of Lemma~\ref{lem-phi}}.~  
Since $(V^{-1})'(t) = 1 / V'(V^{-1}(t))$ for $t > 1$, we obtain
\begin{equation}
\phi'(t) 
= \kappa V''(V^{-1}(t)) \cdot (V^{-1})'(t)
= \kappa {V''(V^{-1}(t)) \over V'(V^{-1}(t))},\qquad t > 1.
\label{eqn-phi'(t)}
\end{equation}
It follows from (\ref{eqn-phi'(t)}), $V'(x) > 0$ and $V''(x) \ge 0$
for $x > 0$ that
\[
\phi'(t) \ge 0,\qquad t > 1,
\]
which shows that $\phi$ is nondecreasing. It also follows from
(\ref{eqn-phi'(t)}) and condition (iv) of
Assumption~\ref{assumpt-moment-cond} that $\phi'$ is nonincreasing,
which implies that $\phi$ is concave. In addition, since
$\lim_{x\to\infty}V'(x) = \infty$ (see condition (iii) of
Assumption~\ref{assumpt-moment-cond}), we have $\lim_{x\to\infty}V(x)
= \infty$ and thus $\lim_{t\to\infty}V^{-1}(t) = \infty$. Therefore,
(\ref{eqn-phi'(t)}) together with (\ref{lim-V''(x)/V'(x)}) yields
$\lim_{t\to\infty}\phi'(t) = 0$.  \qed

\medskip

The following theorem is immediate from
Theorem~\ref{thm-extended-geo2} with $\widetilde{\vc{P}} = \vc{P}$ and
$\phi \circ \vc{v}(k)=\kappa V'(k)\vc{e}$ (see (\ref{eqn-kappa*V'(t)})).

\medskip

\begin{theorem}\label{thm-special}
Suppose that all the conditions of Lemma~\ref{lem-App-01} are
satisfied.  Furthermore, suppose that $\vc{P}^M(K;0)\vc{e} > \vc{0}$,
and fix $B \in (0,\infty)$ such that $B\vc{P}^M(K;0)\vc{e} \ge
b\vc{e}$, where $b $ is given in (\ref{defn-widetilde{b}}). Under
these conditions, we have
\begin{equation}
\left\| \trunc{\vc{\pi}}_n - \vc{\pi} \right\|
\le
{8\{ c_{\phi,B}(1) \}^{-1} \over r_{\phi} \circ c_{\phi,B}(m-1)} (V(1) + B)
+ {2mMbd \over \kappa V'(n)}
\qquad \mbox{for all}~m,n\in\bbN,
\label{bound-GIG1-type}
\end{equation}
where the composite function $r_{\phi} \circ\, c_{\phi,B}$ is given by
(\ref{defn-r_{phi}}), (\ref{defn-c_{phi,B}}) and
(\ref{defn-tilde{phi}}). In addition, if $K=0$, then
\begin{equation}
\left\| \trunc{\vc{\pi}}_n - \vc{\pi} \right\|
\le
{8 \over r_{\phi}(m-1)} V(1)
+ {2mMbd \over \kappa V'(n)}
\qquad \mbox{for all}~m,n\in\bbN.
\label{bound-GIG1-type-K=0}
\end{equation}
\end{theorem}

\medskip

\begin{rem} 
Recall that $\lim_{x\to\infty}V'(x) = \infty$ (see condition (iii) of
Assumption~\ref{assumpt-moment-cond}) and $\lim_{t\to\infty}V^{-1}(t)
= \infty$ (see the proof of Lemma~\ref{lem-phi}). Therefore,
$\lim_{t\to\infty}\phi(t) = \lim_{t\to\infty} \kappa V'(V^{-1}(t)) =
\infty$, which leads to $\lim_{x\to\infty}r_{\phi}(x) = \infty$, as
stated in Remark~\ref{rem-main-thm}.  Consequently, we can choose $m,n
\in \bbN$ such that the error bounds (\ref{bound-GIG1-type}) and
(\ref{bound-GIG1-type-K=0}) are reduced to any desired value less than
two.
\end{rem}

\subsection{Special case}

In this subsection, we consider the LC-block-augmented truncation of a
special block-monotone GI/G/1-type Markov chain, for which we
establish an error bound with specified parameters in accordance with
Theorem~\ref{thm-special}. To this end, in addition to
Assumption~\ref{assumpt-GI/G/1-type}, we assume that $\vc{P}$ in
(\ref{GIG1-type-P}) is reduced to
\begin{equation}
\vc{P}
=
\left(
\begin{array}{ccccc}
\underline{\vc{A}}(0)  & \vc{A}(1)  & \vc{A}(2)  & \vc{A}(3) & \cdots
\\
\underline{\vc{A}}(-1) & \vc{A}(0)  & \vc{A}(1)  & \vc{A}(2) & \cdots
\\
\underline{\vc{A}}(-2) & \vc{A}(-1) & \vc{A}(0)  & \vc{A}(1) & \cdots
\\
\underline{\vc{A}}(-3) & \vc{A}(-2) & \vc{A}(-1) & \vc{A}(0) & \cdots
\\
\vdots     & \vdots     &  \vdots    & \vdots    & \ddots
\end{array}
\right),
\label{MG1-type-P}
\end{equation}
with
\begin{eqnarray}
\vc{A}(k) 
&=& 
\left(
\begin{array}{cc}
0 & 2^{k-1}
\\
2^{k-1} & 0
\end{array}
\right),
\qquad k \le -1,
\label{defn-A(k)-special-01}
\\
\vc{A}(k) 
&=& 
{1 \over 2}\left(
\begin{array}{cc}
0 & \dm{(k+1)^{-\beta_1}\over \zeta(\beta_1)}
\\
\dm{(k+1)^{-\beta_2}\over \zeta(\beta_2)} & 0
\end{array}
\right),
\qquad k \in \bbZ_+,
\label{defn-A(k)-special-02}
\end{eqnarray}
where $2 < \beta_1 < \beta_2$ and $\zeta(\,\cdot\,)$ denotes the
Riemann zeta function.  The matrix $\vc{P}$ in (\ref{MG1-type-P}) can
be regarded as the transition probability matrix of an irreducible
reflected Markov additive process \cite[Chapter XI, Section
  2e]{Asmu03}.  It is easy to see that the stationary probability
vector $\vc{\varpi}$ of $\vc{A}=\sum_{k=-\infty}^{\infty}\vc{A}(k)$ is
given by
\begin{equation}
\vc{\varpi} 
= \left(
\begin{array}{cc}
\dm{1 \over 2} & \dm{1 \over 2}
\end{array}
\right).
\label{eqn-varpi-special}
\end{equation}

For convenience, let
\begin{eqnarray}
\qquad\qquad
a_1(k) = 
\left\{
\begin{array}{ll}
2^{k-1}, & k \le -1,
\\
\rule{0mm}{7mm}
\dm{1 \over 2}{(k+1)^{-\beta_1}\over \zeta(\beta_1)}, & k \in \bbZ_+,
\end{array}
\right.
\quad
a_2(k) = 
\left\{
\begin{array}{ll}
2^{k-1}, & k \le -1,
\\
\rule{0mm}{7mm}
\dm{1 \over 2}{(k+1)^{-\beta_2} \over \zeta(\beta_2)}, & k \in \bbZ_+.
\end{array}
\right. 
\label{defn-a_1(k)-a_2(k)}
\end{eqnarray}
It then follows from (\ref{defn-A(k)-special-01}),
(\ref{defn-A(k)-special-02}) and (\ref{defn-a_1(k)-a_2(k)}) that
\[
\vc{A}(k)\vc{e} 
=
\left(
\begin{array}{cc}
a_1(k)
\\
a_2(k) 
\end{array}
\right),\qquad k \in \bbZ,
\]
and thus
\begin{equation}
\sum_{k=-\infty}^{\infty}k\vc{A}(k)\vc{e} 
=
\left(
\begin{array}{cc}
\overline{a}_1
\\
\overline{a}_2 
\end{array}
\right),
\label{eqn-sum-kA(k)e}
\end{equation}
where 
\begin{eqnarray}
\overline{a}_1 
&:=& \sum_{k=-\infty}^{\infty} k a_1(k)
= {1 \over 2} \left({\zeta(\beta_1 - 1) \over \zeta(\beta_1)} - 3 \right), 
\label{eqn-overline{a}_1}
\\
\overline{a}_2 
&:=& \sum_{k=-\infty}^{\infty} k a_2(k)
= {1 \over 2} \left({\zeta(\beta_2 - 1) \over \zeta(\beta_2)} - 3 \right).
\label{eqn-overline{a}_2}
\end{eqnarray}
Substituting (\ref{eqn-varpi-special}) and (\ref{eqn-sum-kA(k)e}) into
(\ref{defn-sigma}) and using (\ref{eqn-overline{a}_1}) and
(\ref{eqn-overline{a}_2}) yield
\begin{eqnarray}
\sigma 
&=& {1 \over 2} (\overline{a}_1 + \overline{a}_2)
= {1 \over 4} 
\left(
{\zeta(\beta_1 - 1) \over \zeta(\beta_1)}
+
{\zeta(\beta_2 - 1) \over \zeta(\beta_2)} - 6 
\right).
\label{eqn-sigma-special}
\end{eqnarray}
Note here that $1 < \zeta(s_2) < \zeta(s_1) < \pi^2/6$ for $2 < s_1 <
s_2$, which leads to
\begin{equation}
1 < {\zeta(\beta_1 - 1) \over \zeta(\beta_1)} < {\pi^2 \over 6} < {5 \over 3},\qquad
1 < {\zeta(\beta_2 - 1) \over \zeta(\beta_2)} < {\pi^2 \over 6} < {5 \over 3}.
\label{ineqn-zeta}
\end{equation}
Therefore, from (\ref{eqn-overline{a}_1}) and
(\ref{eqn-overline{a}_2}), we have
\[
\overline{a}_1 < 0,\qquad \overline{a}_2 < 0.
\]
Applying these two inequalities to (\ref{eqn-sigma-special}), we
obtain $\sigma < 0$. i.e., Assumption~\ref{assumpt-stability-cond}
holds. As a result, $\vc{P}$ in (\ref{MG1-type-P}) has the unique
stationary probability vector $\vc{\pi}$.
 
Recall that Theorem~\ref{thm-special} holds under all the conditions
of Lemma~\ref{lem-App-01}. Therefore, to fulfill the conditions, we
determine parameters $N \in \bbN$, $M \in \bbN$, $\kappa > 0$,
$\varepsilon \in (0,1)$, $\delta_0 > 0$, $K_0 \in \{L,L+1,\dots\}$ and
$K \in \{K_0,K_0+1,\dots\}$.

We begin with $N \in \bbN$ and $M \in \bbN$. In the present special
case, the stochastic matrix $\vc{P}_1$ in (\ref{defn-P_N}) is
expressed as
\begin{equation}
\vc{P}_1
=
\left(
\begin{array}{ccccc}
\underline{\vc{A}}(0) & \vc{A}(1)  & \vc{A}(2)  & \vc{A}(3) & \cdots
\\
\underline{\vc{A}}(-1) & \vc{A}(0)  & \vc{A}(1)  & \vc{A}(2) & \cdots
\\
\vc{O} & \underline{\vc{A}}(-1) & \vc{A}(0)  & \vc{A}(1) & \cdots
\\
\vc{O} & \vc{O} & \underline{\vc{A}}(-1) & \vc{A}(0) & \cdots
\\
\vdots     & \vdots     &  \vdots    & \vdots    & \ddots
\end{array}
\right).
\label{MG1-type-P-special}
\end{equation}
Note here that (\ref{defn-widetilde{A}(k)}) yields
\begin{equation}
\vc{A}_1(-1) = \underline{\vc{A}}(-1)
= \sum_{k=-\infty}^{-1}\vc{A}(k),
\qquad \vc{A}_1(k) = \vc{A}(k),\quad k \in \bbZ_+.
\label{eqn-A_1(k)=A(k)}
\end{equation}
Thus, we have
\begin{eqnarray}
\qquad
\sum_{k = -1}^{\infty}k \vc{A}_1( k)\vc{e}
&=& - \underline{\vc{A}}(-1)\vc{e} 
+ \sum_{k = 1}^{\infty} k \vc{A}(k)\vc{e}
=
\left(
\begin{array}{c}
\dm{1 \over 2}\left(
{\zeta(\beta_1 - 1) \over \zeta(\beta_1)} - 2 
\right)
\\
\rule{0mm}{7mm}
\dm{1 \over 2}\left(
{\zeta(\beta_2 - 1) \over \zeta(\beta_2)} - 2 
\right)
\end{array}
\right)
< \vc{0},
\label{eqn-sum_k*A_1(k)e}
\end{eqnarray}
where the second equality follows from (\ref{defn-A(k)-special-01})
and (\ref{defn-A(k)-special-02}), and where the last inequality
follows from (\ref{ineqn-zeta}). Furthermore, applying
(\ref{eqn-varpi-special}) and (\ref{eqn-sum_k*A_1(k)e}) to
(\ref{defn-sigma_N-02}) with $N=1$, we obtain
\begin{eqnarray}
\sigma_1 
&=& 
{1 \over 4} 
\left(
{\zeta(\beta_1 - 1) \over \zeta(\beta_1)}
+
{\zeta(\beta_2 - 1) \over \zeta(\beta_2)} - 4 
\right) < 0.
\label{eqn-sigma_1}
\end{eqnarray}
The inequalities (\ref{eqn-sum_k*A_1(k)e}) and (\ref{eqn-sigma_1})
imply that (\ref{cond-sigma_N<0}) and (\ref{defn-sigma_N-04}) hold for
$N=M_0=1$.  To proceed, we fix $N = M = M_0 = 1$ and thus $L=MN=1$.

We then consider $\kappa > 0$ and $\varepsilon \in (0,1)$.  Using
(\ref{eqn-A_1(k)=A(k)}), we can reduce (\ref{ineqn-varepsilon}) to
\[
- (1-\varepsilon) \underline{\vc{A}}(-1)\vc{e}
+ (1 + \varepsilon)\sum_{\ell = 1}^{\infty} \ell \vc{A}(\ell)\vc{e}
\le -2\kappa\vc{e}.
\]
Substituting (\ref{defn-A(k)-special-01}) and
(\ref{defn-A(k)-special-02}) into the above inequality, we have
\begin{equation}
\left(
\begin{array}{c}
\dm{1 + \varepsilon \over 2}{\zeta(\beta_1 - 1) \over \zeta(\beta_1)} - 1 
\\
\rule{0mm}{7mm}
\dm{1 + \varepsilon \over 2}{\zeta(\beta_2 - 1) \over \zeta(\beta_2)} - 1
\end{array}
\right)
\le -2\kappa
\left(
\begin{array}{c}
1
\\
1
\end{array}
\right).
\label{ineqn-kappa-epsilon}
\end{equation}
Since the Riemann zeta function $\zeta(s)$ is log-convex for $s > 1$
(see, e.g., \cite{Gut05}),
\begin{eqnarray*}
{\rmd \over \rmd s}
\left( { \zeta(s-1) \over \zeta(s) } \right)
&=&  {\zeta(s-1) \over \zeta(s)}
\left(
{ \zeta'(s-1) \over \zeta(s-1) }
-
{ \zeta'(s) \over \zeta(s) }
\right)
\nonumber
\\
&=& {\zeta(s-1) \over \zeta(s)}
\left\{
{\rmd \over \rmd s}
\log \zeta(s-1)
-
{\rmd \over \rmd s}
\log \zeta(s)
\right\}
\le 0,
\end{eqnarray*}
which leads to
\begin{equation}
{\zeta(\beta_2 - 1) \over \zeta(\beta_2)} 
\le
{\zeta(\beta_1 - 1) \over \zeta(\beta_1)}.
\label{ineqn-ol{a}_1<ol{a}_2}
\end{equation}
The inequalities (\ref{ineqn-kappa-epsilon}) and
(\ref{ineqn-ol{a}_1<ol{a}_2}) imply that (\ref{ineqn-kappa-epsilon})
holds if
\begin{equation}
{1 + \varepsilon \over 2}{\zeta(\beta_1 - 1) \over \zeta(\beta_1)} - 1 = -2\kappa.
\label{eqn-kappa-epsilon}
\end{equation}
Therefore, we fix $\kappa > 0$ and $\varepsilon \in (0,1)$ such that
\begin{eqnarray}
\kappa 
&=& {1 \over 4}
\left(1 - {\zeta(\beta_1 - 1) \over 2\zeta(\beta_1)} \right)
> {1 \over 24},
\label{defn-kappa}
\\
\varepsilon
&=& {1 \over 2}
\left({2\zeta(\beta_1) \over \zeta(\beta_1 - 1)} - 1\right)
\in \left({1 \over 10}, {1 \over 2} \right),
\label{defn-varepsilon}
\end{eqnarray}
which satisfy (\ref{eqn-kappa-epsilon}) and thus
(\ref{ineqn-kappa-epsilon}).

Next, we determine $\delta_0 > 0$, $K_0 \in \{L,L+1,\dots\}$ (with
$L=1$) and the function $V$ in
Assumption~\ref{assumpt-moment-cond}. From
(\ref{defn-A(k)-special-02}), we have
\[
\lim_{k\to\infty}{\vc{A}(k) \over k^{-\beta_1}}
= \left(
\begin{array}{cc}
0 & \dm{1 \over 2\zeta(\beta_1)}
\\
0 & 0
\end{array}
\right),
\]
which corresponds to the case discussed in
Appendix~\ref{appendix-Polynomial}.  Thus, we fix $\alpha = 0$,
$\beta_0 \in (1,\beta_1 -1)$ arbitrarily and
\begin{equation}
V(x) = (x+x_0)^{\beta_0},\qquad x \ge 0,
\label{eqn-V-Pareto-again}
\end{equation}
with $x_0 > 0$. It then follows that
Assumption~\ref{assumpt-moment-cond} holds for $\alpha = 0$ (see
Appendix~\ref{appendix-Polynomial}).  It also follows from
(\ref{eqn-V-Pareto-again}), $\alpha = 0$ and $L=1$ that
(\ref{ineqn-K_0-01}) and (\ref{ineqn-K_0-02}) are reduced to
\begin{eqnarray}
{ (\delta_0 + 1)k + x_0 \over k+x_0 }  
&\le& (1 + \varepsilon)^{1/(\beta_0-1)},
\qquad  k \ge K_0+1,
\label{ineqn-K_0-01-special}
\\
{ k+x_0-1 \over k+x_0 }
&\ge& (1 - \varepsilon)^{1/(\beta_0-1)},
\qquad  k \ge K_0+1.
\label{ineqn-K_0-02-special}
\end{eqnarray}
Note that
\[
\sup_{k\ge0}{ (\delta_0 + 1)k + x_0 \over k+x_0 } 
= \delta_0 + 1,
\qquad
\inf_{k\ge0}{ k+x_0-1 \over k+x_0 }
= 1 - {1 \over x_0}.
\]
Thus, we fix $K_0 \in \bbN$, $\delta_0 >0$ and $x_0 > 0$ such that
\begin{eqnarray}
K_0
&=& 1,
\qquad
\delta_0 
= (1 + \varepsilon)^{1/(\beta_0-1)} - 1,
\nonumber
\\
x_0
&=& {1 \over 1 - (1 - \varepsilon)^{1/(\beta_0-1)}},
\label{defn-x_0}
\end{eqnarray}
where $\varepsilon$ is given in (\ref{defn-varepsilon}). It is easy to
see that (\ref{ineqn-K_0-01-special}) and (\ref{ineqn-K_0-02-special})
hold.

Finally, we discuss the remaining parameter $K \in
\{K_0,K_0+1,\dots\}$ (with $K_0=1$).  Let $\rho$ denote
\begin{equation}
\rho = \max(1 + 1/\delta_0,x_0).
\label{defn-rho}
\end{equation}
Let $C_1$ and $C_2$ denote
\begin{eqnarray}
C_1
= {\rho^{\beta_0} \delta_0^{-\beta_1+\beta_0+1} 
\over 2\beta_0 (\beta_1 - \beta_0 - 1) \zeta(\beta_1)},
\qquad
C_2
= {\rho^{\beta_0} \delta_0^{-\beta_2+\beta_0+1} 
\over 2\beta_0 (\beta_2 - \beta_0 - 1) \zeta(\beta_2)},
\label{defn-C_1-C_2}
\end{eqnarray}
respectively. Furthermore, fix
\begin{equation}
K 
= \max\left( 
\left\lfloor { C_1 \over \kappa } \right\rfloor^{1/(\beta_1-2)},
\left\lfloor { C_2 \over \kappa } \right\rfloor^{1/(\beta_2-2)}
\right).
\label{eqn-K}
\end{equation}
It then holds that
\begin{equation}
{1 \over V'(k)}
\sum_{\ell=\lfloor \delta_0 k \rfloor + 1}^{\infty}
V(k + \ell)\vc{A}(\ell)\vc{e}
\le \kappa \vc{e}\quad \mbox{for all $k \ge K+1$},
\label{ineqn-151205-02}
\end{equation}
which is proved in Appendix~\ref{appen-ineqn-151205-02}. Inequality
(\ref{ineqn-151205-02}) together with (\ref{eqn-A_1(k)=A(k)}) implies
that (\ref{ineqn-K_1-02}) holds for $M=N=1$, $\alpha = 0$ and $K$
given in (\ref{eqn-K}).

We have confirmed that all the conditions of Lemma~\ref{lem-App-01}
hold for $N=M=1$. Therefore, Lemma~\ref{lem-App-01} implies that
\begin{eqnarray}
\sum_{\ell=0}^{\infty}\vc{P}_1(k;\ell)V(\ell)\vc{e} -V(k)\vc{e}
\le -\kappa V'(k)\vc{e},
\qquad k \ge K + 1.
\label{ineqn-sum-P(k;l)V(l)}
\end{eqnarray}
We now fix $\vc{v} \in \sfBI_d$ such that
\begin{eqnarray}
\vc{v}(k) = V(k) \vc{e} = (k+x_0)^{\beta_0}\vc{e}, \qquad k \in \bbZ_+,
\label{eqn-v(k)}
\end{eqnarray}
where $x_0$ is given in (\ref{defn-x_0}). We also fix $\phi$ as in
(\ref{defn-tilde{phi}}). Thus, (\ref{eqn-kappa*V'(t)}) holds. Applying
(\ref{eqn-kappa*V'(t)}) and (\ref{eqn-v(k)}) to
(\ref{ineqn-sum-P(k;l)V(l)}) yields
\begin{equation}
\sum_{\ell=0}^{\infty}\vc{P}_1(k;\ell)\vc{v}(\ell) - \vc{v}(k)
\le - \phi \circ \vc{v}(k),
\qquad k \ge K + 1.
\label{ineqn-sum-P(k;l)v(l)}
\end{equation}
Consequently, we can establish the drift conditions
(\ref{ineqn-P_N_widetilde{v}}) and (\ref{ineqn-P_widetilde{v}}) with
$N=M=1$ once we find an upper bound $b\vc{e}$ for
$\{\sum_{\ell=0}^{\infty}\vc{P}_1(k;\ell)V(\ell)\vc{e} - V(k)\vc{e} +
\kappa V'(k)\vc{e}; k=0,1,\dots,K\}$.

We specify the parameter $b$. 
It follows from (\ref{defn-A(k)-special-01}),
(\ref{defn-A(k)-special-02}) and (\ref{MG1-type-P-special}) that
\begin{eqnarray}
\lefteqn{
\sum_{\ell=0}^{\infty}\vc{P}_1(0;\ell)V(\ell)\vc{e} 
- V(0)\vc{e} + \kappa V'(0)\vc{e}
}
\quad &&
\nonumber
\\
&=& \underline{\vc{A}}(0) V(0) \vc{e}
+ \sum_{\ell = 1}^{\infty}\vc{A}(\ell)V(\ell)\vc{e}
- V(0)\vc{e}
+ \kappa V'(0)\vc{e}
\nonumber
\\
&\le& \sum_{\ell = 1}^{\infty}\vc{A}(\ell)V(\ell)\vc{e}
+ \kappa V'(0)\vc{e},
\label{add-eqn-160309-01}
\end{eqnarray}
where the inequality is due to $\underline{\vc{A}}(0)\vc{e}
\le \vc{e}$. Similarly, for $k=1,2,\dots,K$,
\begin{eqnarray}
\lefteqn{
\sum_{\ell=0}^{\infty}\vc{P}_1(k;\ell)V(\ell)\vc{e} 
- V(k)\vc{e} + \kappa V'(k)\vc{e}
}
\quad &&
\nonumber
\\
&=& \underline{\vc{A}}(-1) V(k-1) \vc{e}
+ \sum_{\ell = 0}^{\infty}\vc{A}(\ell)V(k+\ell)\vc{e}
- V(k)\vc{e}
+ \kappa V'(k)\vc{e}
\nonumber
\\
&\le& \sum_{\ell = 0}^{\infty}\vc{A}(\ell)V(k+\ell)\vc{e}
+ \kappa V'(k)\vc{e},
\label{add-eqn-160309-02}
\end{eqnarray}
where the inequality holds because $\underline{\vc{A}}(-1)\vc{e} \le
\vc{e}$ and $V(k-1) < V(k)$ for all $k \in \bbN$. Since $V$ and $V'$
is nondecreasing, we have, from (\ref{add-eqn-160309-01}) and
(\ref{add-eqn-160309-02}),
\begin{eqnarray*}
\lefteqn{
\sum_{\ell=0}^{\infty}\vc{P}_1(k;\ell)V(\ell)\vc{e} 
- V(k)\vc{e} + \kappa V'(k)\vc{e}
}
\quad &&
\nonumber
\\
&\le& \sum_{\ell = 0}^{\infty}\vc{A}(\ell)V(K+\ell)\vc{e}
+ \kappa V'(K)\vc{e},
\qquad k = 0,1,\dots,K.
\end{eqnarray*}
Applying (\ref{defn-A(k)-special-02}) and (\ref{eqn-V-Pareto-again})
to the right hand side of the above inequality, we obtain
\begin{equation}
\sum_{\ell=0}^{\infty}\vc{P}_1(k;\ell)V(\ell)\vc{e} 
- V(k)\vc{e} + \kappa V'(k)\vc{e}
\le b\vc{e},
\qquad k = 0,1,\dots,K,
\label{ineqn-sum-P(k;l)V(k+l)}
\end{equation}
where
\begin{eqnarray}
b
&=& \max\left( 
{1 \over 2\zeta(\beta_1)}
\sum_{\ell=0}^{\infty}{(K+\ell+x_0)^{\beta_0} \over (\ell+1)^{\beta_1}},
{1 \over 2\zeta(\beta_2)}
\sum_{\ell=0}^{\infty}{(K+\ell+x_0)^{\beta_0} \over (\ell+1)^{\beta_2}}
\right)
\nonumber
\\
&& {} 
+ \kappa \beta_0 (k + x_0)^{\beta_0 - 1}.
\label{eqn-b-special}
\end{eqnarray}
Substituting (\ref{eqn-kappa*V'(t)}) and (\ref{eqn-v(k)}) into
(\ref{ineqn-sum-P(k;l)V(k+l)}) results in
\begin{equation}
\sum_{\ell=0}^{\infty}\vc{P}_1(k;\ell) \vc{v}(\ell) 
- \vc{v}(k) + \phi \circ \vc{v}(k) \le b\vc{e},
\qquad k = 0,1,\dots,K.
\label{ineqn-sum-P(k;l)v(k+l)}
\end{equation}
Combining (\ref{ineqn-sum-P(k;l)v(k+l)}) with
(\ref{ineqn-sum-P(k;l)v(l)}) leads to
\begin{equation}
\vc{P}_1\vc{v} \le \vc{v} - \phi \circ \vc{v}
+ b \vc{1}_K.
\label{ineqn-P_1v}
\end{equation}
Recall here that $\vc{P} \prec_d \vc{P}_1 \in \sfBM_d$ and $\vc{v} \in
\sfBI_d$. Thus, $\vc{P}\vc{v} \le \vc{P}_1\vc{v}$ (due to
(\ref{ineqn-P^Mv}) with $N=M=1$). This inequality and
(\ref{ineqn-P_1v}) yield the drift condition on $\vc{P}$:
\[
\vc{P}\vc{v} \le \vc{v} - \phi \circ \vc{v} + b \vc{1}_K.
\]
In addition, $\vc{P}(k;0)\vc{e}=\underline{\vc{A}}(-k)\vc{e} > \vc{0}$
for all $k \in \bbN$, which follows from (\ref{MG1-type-P}),
(\ref{defn-A(k)-special-01}) and (\ref{defn-A(k)-special-02}).
 
We have shown that all the conditions of Theorem~\ref{thm-special} are
satisfied. We now fix
\begin{equation}
B = 2^{K}b,
\label{eqn-B=2^Kb}
\end{equation}
where $b$ is given in (\ref{eqn-b-special}). 
It then follows from (\ref{MG1-type-P}),
(\ref{defn-A(k)-special-01}) and (\ref{defn-A(k)-special-02}) that
\[
B\vc{P}(K;0)\vc{e} 
= 2^{K}b\underline{\vc{A}}(-K)\vc{e} 
= b\vc{e}.
\]
As a result, using Theorem~\ref{thm-special}, we can derive an error
bound for the special case considered here.

In what follows, we present the components of the error bound.  Since
$V^{-1}(t) = t^{1/\beta_0} - x_0$ for $t \ge x_0^{\beta_0}$, the
function $\phi$ in (\ref{defn-tilde{phi}}) is expressed as
\begin{equation}
\phi(t) = \kappa \beta_0 t^{1 -1/\beta_0},
\qquad t \ge 1.
\label{eqn-phi(t)-special}
\end{equation}
Therefore, the function $H_{\phi}$ in (\ref{defn-H_{phi}}) is given by
\begin{equation}
H_{\phi}(x) = \int_1^x {y^{-1 + 1/\beta_0} \over  \kappa \beta_0}\rmd y
= \kappa^{-1}(x^{1/\beta_0} - 1),
\qquad x \ge 1.
\label{eqn-H_{phi}-special}
\end{equation}
Substituting (\ref{eqn-phi(t)-special}) and
(\ref{eqn-H_{phi}-special}) into (\ref{defn-r_{phi}}), we have
\begin{equation}
r_{\phi}(x) = \kappa \beta_0 (\kappa x + 1)^{\beta_0-1},
\qquad x \ge 0.
\label{eqn-r_{phi}-special}
\end{equation}
Furthermore, applying (\ref{eqn-B=2^Kb}) and (\ref{eqn-phi(t)-special})
 to (\ref{defn-c_{phi,B}}), we obtain
\begin{equation}
c_{\phi,B}(x)
= \breve{c}  x,
\qquad x \ge 0,
\label{eqn-c_{phi,B}(x)}
\end{equation}
where
\begin{equation}
\breve{c} = ( 2^K b + 1)^{-1 + 1/\beta_0}.
\label{defn-breve{c}}
\end{equation}
Combining (\ref{eqn-r_{phi}-special}) and (\ref{eqn-c_{phi,B}(x)})
yields
\begin{equation}
r_{\phi} \circ c_{\phi,B}(x)
= \kappa \beta_0 (\kappa\breve{c}x + 1)^{\beta_0- 1},
\qquad x \ge 0.
\label{eqn-c_{phi,B}(1)}
\end{equation}
Consequently, letting $M=1$ and $d=2$ in (\ref{bound-GIG1-type}) and
using (\ref{eqn-V-Pareto-again}), (\ref{eqn-B=2^Kb}),
(\ref{eqn-c_{phi,B}(x)}) and (\ref{eqn-c_{phi,B}(1)}), we obtain
\begin{eqnarray}
\left\| \trunc{\vc{\pi}}_n - \vc{\pi} \right\|
&\le&
{
8\breve{c}^{-1}
\over \kappa \beta_0 \{\kappa\breve{c}(m-1) + 1\}^{\beta_0-1}
} 
\left\{ (1+x_0)^{\beta_0} + 2^K b\right\}
\nonumber
\\
&& {}
+ {4mb \over \kappa\beta_0(n+x_0)^{\beta_0-1}}
\qquad \mbox{for all}~m,n\in\bbN,
\label{bound-GIG1-type-special}
\end{eqnarray}
where $\kappa$, $x_0$, $K$, $b$ and $\breve{c}$ are given by
(\ref{defn-kappa}), (\ref{defn-x_0}), (\ref{eqn-K}),
(\ref{eqn-b-special}) and (\ref{defn-breve{c}}), respectively.

Finally, using the obtained bound (\ref{bound-GIG1-type-special}), we
determine a truncation parameter $n \in \bbN$ such that $\left\|
\trunc{\vc{\pi}}_n - \vc{\pi} \right\|$ is within a given tolerance
$\calE \in (0,2)$, i.e.,
\begin{equation}
\left\| \trunc{\vc{\pi}}_n - \vc{\pi} \right\| \le \calE.
\label{error-tolerance}
\end{equation}
Let $m_0$ and $n_0$ denote
\begin{eqnarray*}
\qquad
m_0
&=&
\min\left\{
m \in \bbN;
{
8\breve{c}^{-1}
\over \kappa \beta_0 \{\kappa\breve{c}(m-1) + 1\}^{\beta_0-1}
} 
\left\{ (1+x_0)^{\beta_0} + 2^K b\right\}
\le {\calE \over 2}
\right\},
\\
n_0
&=&
\min\left\{
n \in \bbN;
{4m_0b \over \kappa\beta_0(n+x_0)^{\beta_0-1}}
\le {\calE \over 2}
\right\},
\end{eqnarray*}
respectively.
We then have
\begin{eqnarray}
m_0
&=& \left \lceil 
{1 \over \kappa\breve{c}}
\left(
\left[
{
16\breve{c}^{-1}
\over \kappa \beta_0 \calE
} 
\left\{ (1+x_0)^{\beta_0} + 2^K b\right\}
\right]^{1/(\beta_0-1)} - 1 
\right)
\right\rceil
+1,
\label{defn-m_0}
\\
n_0
&=& \max\left(1,\left \lceil 
\left( {8m_0b \over \kappa\beta_0\calE} \right)^{1/(\beta_0-1)} - x_0
\right\rceil
\right).
\label{defn-n_0}
\end{eqnarray}
Substituting $m_0$ in (\ref{defn-m_0}) and $n_0$ in (\ref{defn-n_0})
into (\ref{bound-GIG1-type-special}) yields (\ref{error-tolerance}).

\section{Concluding remarks}\label{sec-remarks}

This paper studied the estimation of the total variation distance
between the stationary probability vectors of a discrete-time
block-structured Markov chain and its LC-block-augmented
truncation. The main contribution of this paper is to present a
total-variation-distance error bound for the stationary probability
vector of the LC-block-augmented truncation under the assumption that
the original Markov chain is block monotone and satisfies the {\it
  subgeometric} drift condition proposed in \cite{Douc04}. This paper
is complementary to the author's previous study \cite{Masu15}, which
considered discrete-time block-monotone Markov chains satisfying the
{\it geometric} drift condition. The author \cite{Masu15-LAA} also
considered continuous-time block-monotone Markov chains with
exponential ergodicity and derived a total-variation-distance error
bound for the stationary probability vector of the LC-block-augmented
truncation.  The present study and the author's previous ones
\cite{Masu15,Masu15-LAA} depend on the notion of block monotonicity.

Recently, without block monotonicity (including monotonicity), the
author \cite{Masu16} established computable upper bounds for the
absolute difference between the time-averaged functionals of a
continuous-time block-structured Markov chain and its
LC-block-augmented truncation under the assumption that the original
Markov chain satisfies the $\vc{f}$-modulated drift condition (see
\cite[Condition 1,1]{Masu16}). The $\vc{f}$-modulated drift condition
includes the geometric (or exponential in continuous time) drift
condition and the subgeometric (or subexponential in continuous time)
drift condition proposed in \cite{Douc04} as special cases. Therefore,
the error bounds in \cite{Masu16} are widely applicable, though they
are more computationally costly than those in this paper and
\cite{Masu15,Masu15-LAA} with block monotonicity.


\appendix

\section{Examples of function $V$ in Assumption~\ref{assumpt-moment-cond}}\label{appen-examp-F}

We begin with the following lemma.

\medskip

\begin{lemma}\label{lem-sufficient-cond-V}
Equation (\ref{lim-V'(k)}) holds if there exists some $\alpha \in
[0,1)$ such that at least one of (\ref{lim-V'(k)-(i)}) and
  (\ref{lim-V'(k)-(ii)}) below is true for any $\delta > 0$.
\begin{eqnarray}
&&
\lim_{k\to\infty}
{V(k) \over V'(k)}
\sum_{\ell=\lfloor \delta k^{1 - \alpha} \rfloor + 1}^{\infty}
V(\ell)\vc{A}_N^{\ast M}(\ell)\vc{e} 
= \vc{0},
\label{lim-V'(k)-(i)}
\\
&&
\limsup_{\ell\to\infty}{V(\delta \ell) \over V(\ell)} < \infty
\quad \mbox{and} \quad
\sum_{\ell=0}^{\infty} V(\ell^{1/(1-\alpha)})\vc{A}_N^{\ast M}(\ell)\vc{e} 
~\mbox{is finite},
\label{lim-V'(k)-(ii)}
\end{eqnarray}
where $V: [0,\infty) \to [1,\infty)$ is an increasing, convex and
    log-concave function that satisfies conditions (i)--(v) of
    Assumption~\ref{assumpt-moment-cond}.
\end{lemma}

\medskip

\begin{proof}
Since $V$ is log-concave and $\log V(0) \ge 0$, we have, for $x,y \ge
0$,
\begin{eqnarray*}
\log V(x) + \log V(y)
&\ge& 
  \left({x \over x+y} \log V(x+y) + {y \over x+y} \log V(0) \right)
\nonumber
\\
&& {} 
+ \left({y \over x+y} \log V(x+y) + {x \over x+y} \log V(0) \right)
\nonumber
\\
&\ge& \log V(x+y),
\end{eqnarray*}
and thus $V(x+y) \le V(x)V(y)$ for $x,y \ge 0$. Using
this inequality, we have
\[
\sum_{\ell=\lfloor \delta k^{1 - \alpha} \rfloor + 1}^{\infty}
V(k + \ell)\vc{A}_N^{\ast M}(\ell)\vc{e} 
\le
V(k)
\sum_{\ell=\lfloor \delta k^{1 - \alpha} \rfloor + 1}^{\infty}
V(\ell)\vc{A}_N^{\ast M}(\ell)\vc{e}. 
\]
Therefore, (\ref{lim-V'(k)-(i)}) implies (\ref{lim-V'(k)}).

Next, we prove that (\ref{lim-V'(k)-(ii)}) implies
(\ref{lim-V'(k)}). To this end, we suppose that (\ref{lim-V'(k)-(ii)})
holds. It then follows that, for any $c > 0$, $\sum_{\ell=0}^{\infty}
V(c\ell^{1/(1-\alpha)})\vc{A}_N^{\ast M}(\ell)\vc{e}$ is finite and thus
\begin{equation}
\lim_{k\to\infty}
\sum_{\ell=\lfloor \delta k^{1 - \alpha} \rfloor + 1}^{\infty}
V(c\ell^{1/(1-\alpha)})\vc{A}_N^{\ast M}(\ell)\vc{e} 
= \vc{0}.
\label{eqn-01}
\end{equation}
In addition, since $V$ is increasing and convex,
\begin{equation}
\liminf_{k\to\infty}V'(k) > 0,
\label{add-eqn-0526-02}
\end{equation}
and, for any $\delta > 0$,
\begin{eqnarray}
\sum_{\ell=\lfloor \delta k^{1 - \alpha} \rfloor + 1}^{\infty}
V(k + \ell)\vc{A}_N^{\ast M}(\ell)\vc{e}
&\le& \sum_{\ell=\lfloor \delta k^{1 - \alpha} \rfloor + 1}^{\infty}
V(\ell+(\ell/\delta)^{1/(1 - \alpha)})\vc{A}_N^{\ast M}(\ell)\vc{e}
\nonumber
\\
&\le& \sum_{\ell=\lfloor \delta k^{1 - \alpha} \rfloor + 1}^{\infty}
V( c_{\delta}\ell^{1/(1 - \alpha)} )\vc{A}_N^{\ast M}(\ell)\vc{e}
\qquad k \in \bbZ_+,
\label{add-eqn-0526-01}
\end{eqnarray}
where $c_{\delta} \ge 1 + \delta^{-1/(1 - \alpha)}$. 
Combining (\ref{eqn-01}), (\ref{add-eqn-0526-02}) and (\ref{add-eqn-0526-01}), we have (\ref{lim-V'(k)}). 
\end{proof}

\medskip

Using Lemma~\ref{lem-sufficient-cond-V}, we present the typical
examples of the function $V$ satisfying
Assumption~\ref{assumpt-moment-cond}, by considering the three cases:
\begin{enumerate}
\renewcommand{\labelenumi}{(\alph{enumi})}
\item $\vc{A}(k) \asymp g_1(k):=\exp\{-c k^{\alpha}\}$ for some $c >
  0$ and $0 < \alpha < 1$;
\item $\vc{A}(k) \asymp g_2(k):=k^{-\beta}$ for some $\beta > 2$; and
\item $\vc{A}(k) \asymp g_3(k):=k^{-2} \{\log (k+1)\}^{-\gamma}$ for some $\gamma > 1$,
\end{enumerate}
where we write $\vc{H}(x) \asymp g(x)$ if $\vc{H}$ is a nonnegative
matrix-valued function such that both $\liminf_{x\to\infty} \vc{H}(x)
/ g(x)$ and $\limsup_{x\to\infty} \vc{H}(x) / g(x)$ are finite and not
equal to the zero matrix for a scalar-valued function $g$ that is
eventually nonnegative. For $i=1,2,3$, it follows from $\vc{A}(k)
\asymp g_i(k)$ and (\ref{defn-widetilde{A}(k)}) that $\vc{A}_N(k)
\asymp g_i(k)$. Thus, $\vc{A}_N^{\ast M}(k) \asymp g_i(k)$ for
$i=1,2,3$, which can be readily proved by using the extensions of
\cite[Proposition~A.2.6]{Kimu13} to the upper and lower limits.

We will see later, from the examples of the function $V$, that the
decay of the error bound (\ref{bound-GIG1-type}) is {\it moderately}
exponential (i.e., heavy-tailed Weibull-like) in Case (a); polynomial
in Case (b); and logarithmic in Case (c), as the truncation parameter
$n$ increases. Therefore, Cases (a), (b) and (c) are called {\it
  moderately exponential case}, {\it polynomial case}, and {\it
  logarithmic case}, respectively.

\subsection{Moderately exponential case}

We suppose that $\vc{A}(k) \asymp g_1(k)$, i.e., $\vc{A}(k) \asymp
\exp\{-c k^{\alpha}\}$ for some $c > 0$ and $0 < \alpha < 1$. We then
fix $V$ such that
\begin{equation}
V(x) = \exp\{c_0 (x + x_0)^{\alpha}\},
\qquad x \ge 0,
\label{eqn-V-Weibul}
\end{equation}
where $0 < c_0 < c$ and $x_0 \ge 1/(\alpha c_0)^{1/\alpha}$.  Clearly,
$V$ is increasing and log-concave, and conditions (i) and (ii) of
Assumption~\ref{assumpt-moment-cond} are satisfied. In what follows,
we confirm that the remaining conditions of
Assumption~\ref{assumpt-moment-cond} are satisfied.

From (\ref{eqn-V-Weibul}), we have
\begin{eqnarray}
V'(x) 
&=& \alpha c_0 (x + x_0)^{\alpha-1}\exp\{c_0 (x + x_0)^{\alpha}\} > 0,
\quad x > 0,
\label{eqn-V'-Weibul}
\\
\qquad V''(x) 
&=&  \alpha c_0 (x + x_0)^{\alpha-2}\exp\{c_0 (x + x_0)^{\alpha}\}
\{\alpha c_0 (x + x_0)^{\alpha} - (1 - \alpha) \}, \quad x > 0,
\label{eqn-V''-Weibul}
\end{eqnarray}
which yield
\begin{eqnarray}
{V'(x) \over V(x)}
&=& \alpha c_0 (x + x_0)^{\alpha-1}, 
\qquad  x > 0,
\label{eqn-V'/V-Weibul}
\\
{V''(x) \over V'(x)}
&=& (x + x_0)^{-1}
\{\alpha c_0 (x + x_0)^{\alpha} - (1 - \alpha)\}, 
\qquad  x > 0. 
\label{eqn-V''/V'-Weibul}
\end{eqnarray}
Equation (\ref{eqn-V'-Weibul}) implies that condition (iii) of
Assumption~\ref{assumpt-moment-cond} holds and that
\begin{eqnarray*}
\lim_{\delta\downarrow0}\lim_{x\to\infty}
{V'(x + \delta x^{1 - \alpha}) \over V'(x)}
&=& \lim_{\delta\downarrow0}\lim_{x\to\infty}
\left(1 + {\delta x^{1 - \alpha} \over x+x_0} \right)^{\alpha - 1}
\nonumber
\\
&& {} \times
\exp\left\{
c_0(x + x_0)^{\alpha}\left[ 
\left( 1 + {\delta x^{1 - \alpha} \over x+x_0} \right)^{\alpha} - 1
\right]
\right\}
\nonumber
\\
&=& \lim_{\delta\downarrow0}\lim_{x\to\infty}
\exp\left\{
c_0(x + x_0)^{\alpha} {\alpha \delta x^{1 - \alpha} \over x + x_0}
\right\}
\nonumber
\\
&=& \lim_{\delta\downarrow0}\lim_{x\to\infty}
\exp\left\{
\alpha  c_0 \delta \left( { x \over x + x_0} \right)^{1 - \alpha}
\right\}
\nonumber
\\
&=& \lim_{\delta\downarrow0}
\exp\left\{ \alpha  c_0 \delta  \right\} = 1,
\end{eqnarray*}
which shows that condition (v) of Assumption~\ref{assumpt-moment-cond}
holds. In addition, since $x_0 \ge 1/(\alpha c_0)^{1/\alpha} >
\{(1-\alpha)/(\alpha c_0)\}^{1/\alpha}$, it follows from
(\ref{eqn-V''-Weibul}) and (\ref{eqn-V''/V'-Weibul}) that, for $x >
0$, $V''(x) > 0$ and $V''(x) / V'(x)$ is nonincreasing, i.e.,
condition (iv) of Assumption~\ref{assumpt-moment-cond} is satisfied.

Finally, we confirm that (\ref{lim-V'(k)}) holds.  It follows from
(\ref{eqn-V-Weibul}), (\ref{eqn-V'/V-Weibul}) and $\vc{A}_N^{\ast
  M}(k) \asymp \exp\{-c k^{\alpha}\}$ that there exists some finite $C
> 0$ such that, for all sufficiently large $k$,
\begin{eqnarray*}
{V(k) \over V'(k)}
\sum_{\ell=\lfloor \delta k^{1 - \alpha} \rfloor + 1}^{\infty}V(\ell)
\vc{A}_N^{\ast M}(\ell)\vc{e} 
&\le& C (k+x_0)^{1 - \alpha}
\sum_{\ell=\lfloor \delta k^{1 - \alpha} \rfloor + 1}^{\infty}
\exp\{-(c-c_0) \ell^{\alpha}\}  \vc{e},
\end{eqnarray*}
which implies that
\[
\lim_{k\to\infty}
{V(k) \over V'(k)}
\sum_{\ell=\lfloor \delta k^{1 - \alpha} \rfloor + 1}^{\infty}
V(\ell)\vc{A}_N^{\ast M}(\ell)\vc{e} 
= \vc{0}.
\]
Combining this and Lemma~\ref{lem-sufficient-cond-V}, we have
(\ref{lim-V'(k)}). Consequently, the function $V$ given in
(\ref{eqn-V-Weibul}) satisfies all the conditions of
Assumption~\ref{assumpt-moment-cond}.

\subsection{Polynomial case}\label{appendix-Polynomial}

We suppose that $\vc{A}(k) \asymp g_2(k) = k^{-\beta}$ for some $\beta
> 2$, and fix $V$ such that
\begin{equation}
V(x) = (x + x_0)^{\beta_0},
\qquad x \ge 0,
\label{eqn-V-Pareto}
\end{equation}
where $1 < \beta_0 < \beta - 1$ and $x_0 > 0$. From
(\ref{eqn-V-Pareto}), we have
\begin{eqnarray}
V'(x) 
&=& \beta_0(x + x_0)^{\beta_0-1} > 0,
\qquad x > 0,
\label{eqn-V'-Pareto}
\\
V''(x) 
&=&  \beta_0 (\beta_0 - 1)(x + x_0)^{\beta_0-2} > 0,
\qquad x > 0.
\nonumber
\end{eqnarray}
Clearly, $V$ is increasing, convex and log-concave, and conditions
(i)--(iv) of Assumption~\ref{assumpt-moment-cond} are
satisfied. From (\ref{eqn-V'-Pareto}), we also obtain
\begin{eqnarray*}
\lim_{\delta\downarrow0}\lim_{x\to\infty}
{V'(x + \delta x) \over V'(x)}
&=& \lim_{\delta\downarrow0} (1 + \delta)^{\beta_0-1} = 1,
\end{eqnarray*}
and thus condition (v) of Assumption~\ref{assumpt-moment-cond} holds
for $\alpha = 0$. Furthermore, it follows from (\ref{eqn-V-Pareto}),
$\alpha = 0$ and $\vc{A}(k) \asymp k^{-\beta}$ that
\begin{eqnarray*}
\lim_{k\to\infty}{ V(\delta k) \over  V(k) } 
&=& \delta^{\beta_0}\qquad \mbox{for any $\delta > 0$},
\\
V(k^{1/(1 - \alpha)})\vc{A}_N^{\ast M}(k)
&\asymp& k^{\beta_0 - \beta}.
\end{eqnarray*}
These equations, together with $\beta_0 - \beta < - 1$, imply
(\ref{lim-V'(k)-(ii)}). Therefore, Lemma~\ref{lem-sufficient-cond-V}
shows that (\ref{lim-V'(k)}) holds for $\alpha = 0$. We have confirmed
that the function $V$ given in (\ref{eqn-V-Pareto}) satisfies all the
conditions of Assumption~\ref{assumpt-moment-cond}.

\subsection{Logarithmic case}

We suppose that $\vc{A}(k) \asymp g_3(k) = k^{-2} \{\log (k+1)
\}^{-\gamma}$ for some $\gamma > 1$, and fix $V$ such that
\begin{equation}
V(x) = (x + x_0) \{\log (x + x_0)\}^{\gamma_0},
\qquad x \ge 0,
\label{eqn-V-Pareto-log}
\end{equation}
where $0 < \gamma_0 < \gamma - 1$ and $x_0 \ge \rme^2$. From
(\ref{eqn-V-Pareto-log}), we have
\begin{eqnarray}
V'(x) 
&=& \{ \log (x + x_0) \}^{\gamma_0-1}
\left[ \log (x + x_0) + \gamma_0 \right] > 0,\quad x > 0,
\label{eqn-V'-Pareto-log}
\\
\qquad~
V''(x) 
&=& \gamma_0 (x + x_0)^{-1} \{ \log (x + x_0) \}^{\gamma_0-2}
[ \log (x + x_0) + \gamma_0 - 1 ] > 0,\quad x > 0.
\label{eqn-V''-Pareto-log}
\end{eqnarray}
Therefore, $V$ is increasing, convex and log-concave and, conditions
(i)--(iii) of Assumption~\ref{assumpt-moment-cond} are satisfied.

In addition, (\ref{eqn-V'-Pareto-log}) and
(\ref{eqn-V''-Pareto-log}) yield
\begin{eqnarray}
\qquad
{V''(x) \over V'(x)}
&=& \gamma_0 (x + x_0)^{-1} 
\left[
\{ \log (x + x_0) \}^{-1}
{ \log (x + x_0) + \gamma_0 - 1 \over \log (x + x_0) + \gamma_0 }
\right],
\quad x > 0.
\label{V''(x)/V'(x)-Pareto-log}
\end{eqnarray}
We now set $y=\log (x + x_0) \ge 2$ and denote by $F(y)$ the part in
the square bracket in the right hand side of
(\ref{V''(x)/V'(x)-Pareto-log}), i.e.,
\[
F(y) = {y + \gamma_0 - 1 \over y(y + \gamma_0)},
\qquad y \ge 2.
\]
We then have
\[
F'(y) = -{(y + \gamma_0 -1)^2 + \gamma_0-1 
\over \{y(y + \gamma_0)\}^2 } < 0
\quad \mbox{for all $y \ge 2$}.
\]
Consequently, $V''(x)/V'(x)$ is nonincreasing for all $x > 0$, i.e.,
condition (iv) of Assumption~\ref{assumpt-moment-cond} holds.

It remains to show that (\ref{lim-V'(k)}) and condition (v) of
Assumption~\ref{assumpt-moment-cond} hold.  It follows from
(\ref{eqn-V'-Pareto-log}) that
\begin{eqnarray*}
\lim_{x\to\infty}
{V'(x + \delta x) \over V'(x)} 
= 1 \quad \mbox{for any $\delta > 0$},
\end{eqnarray*}
which shows that condition (v) of Assumption~\ref{assumpt-moment-cond}
holds for $\alpha = 0$. It also follows from (\ref{eqn-V-Pareto-log}),
$\alpha = 0$ and $\vc{A}_N^{\ast M}(k) \asymp k^{-2} \{\log (k+1)
\}^{-\gamma}$ that
\begin{eqnarray*}
\lim_{k\to\infty}{V(\delta k) \over V(k)}
&=& \delta \qquad \mbox{for any $\delta > 0$},
\\
V(k^{1/(1 - \alpha)})\vc{A}_N^{\ast M}(k) 
&\asymp& k^{-1}\{\log (k+1) \}^{\gamma_0 - \gamma}.
\end{eqnarray*}
Combining these equations with $\gamma_0 - \gamma < -1$, we have
(\ref{lim-V'(k)-(ii)}). Therefore, (\ref{lim-V'(k)}) holds for $\alpha
= 0$ (see Lemma~\ref{lem-sufficient-cond-V}). As a result, the
function $V$ given in (\ref{eqn-V-Pareto-log}) satisfies all the
conditions of Assumption~\ref{assumpt-moment-cond}.

\section{Proof of (\ref{ineqn-151205-02})}\label{appen-ineqn-151205-02}

We first note that (\ref{add-eqn-0526-01}) holds for any $\alpha \in
[0,1)$ and $\delta > 0$.  We then fix $M=N = 1$, $\alpha = 0$, $\delta
  = \delta_0$ and $c_{\delta_0} = \rho \ge 1+1/\delta_0$, where $\rho$
  is given in (\ref{defn-rho}). It then follows from
  (\ref{add-eqn-0526-01}), (\ref{defn-A(k)-special-02}),
  (\ref{eqn-A_1(k)=A(k)}) and (\ref{eqn-V-Pareto-again}) that
\begin{eqnarray}
\lefteqn{
{1 \over V'(k)}
\sum_{\ell=\lfloor \delta_0 k \rfloor + 1}^{\infty}
V(k + \ell)\vc{A}(\ell)\vc{e}
}
\quad &&
\nonumber
\\
&\le& {1 \over V'(k)}
\sum_{\ell=\lfloor \delta_0 k \rfloor + 1}^{\infty}
V( \rho \ell )\vc{A}(\ell)\vc{e}
\nonumber
\\
&\le& {\rho^{\beta_0} \over 2\beta_0(k+x_0)^{\beta_0-1}}
\left(
\begin{array}{c}
\dm\sum_{\ell=\lfloor \delta_0 k \rfloor + 1}^{\infty}
{( \ell + x_0/\rho)^{\beta_0} (\ell + 1)^{-\beta_1} \over \zeta(\beta_1)}
\\
\rule{0mm}{7mm}
\dm\sum_{\ell=\lfloor \delta_0 k \rfloor + 1}^{\infty}
{( \ell + x_0/\rho)^{\beta_0} (\ell + 1)^{-\beta_2} \over \zeta(\beta_2)}

\end{array}
\right)
\nonumber
\\
&\le& {\rho^{\beta_0} \over 2\beta_0 k^{\beta_0-1}}
\left(
\begin{array}{c}
\dm\sum_{\ell=\lfloor \delta_0 k \rfloor + 1}^{\infty}
{( \ell + 1)^{-\beta_1+\beta_0} \over \zeta(\beta_1)}
\\
\rule{0mm}{7mm}
\dm\sum_{\ell=\lfloor \delta_0 k \rfloor + 1}^{\infty}
{( \ell + 1)^{-\beta_2+\beta_0} \over \zeta(\beta_2)}

\end{array}
\right),
\quad k \in \bbN,
\label{ineqn-151205-01}
\end{eqnarray}
where the last inequality follows from $x_0 > 0$ and $x_0/\rho \le
1$ (due to (\ref{defn-rho})). Note here that
\[
\sum_{\ell=\lfloor \delta_0 k \rfloor + 1}^{\infty}
( \ell + 1)^{-\beta}
\le \int_{\lfloor \delta_0 k \rfloor}^{\infty} 
(x+1)^{-\beta} \rmd x
\le { (\delta_0 k )^{-\beta+1} \over \beta - 1},
\qquad
\beta > 1.
\]
Applying this inequality to (\ref{ineqn-151205-01}) yields, for
$k\in\bbN$,
\begin{eqnarray}
{1 \over V'(k)}
\sum_{\ell=\lfloor \delta_0 k \rfloor + 1}^{\infty}
V(k + \ell)\vc{A}(\ell)\vc{e}
&\le& {\rho^{\beta_0} \over 2\beta_0 k^{\beta_0-1}}
\left(
\begin{array}{c}
\dm{
(\delta_0 k )^{-\beta_1+\beta_0+1} 
\over (\beta_1 - \beta_0 - 1)\zeta(\beta_1)
}
\\
\rule{0mm}{7mm}
\dm{
(\delta_0 k )^{-\beta_2+\beta_0+1} 
\over (\beta_2 - \beta_0 - 1)\zeta(\beta_2)
}
\end{array}
\right)
\nonumber
\\
&=& {\rho^{\beta_0} \over 2\beta_0 }
\left(
\begin{array}{c}
\delta_0^{-\beta_1+\beta_0+1}
\dm{k^{-\beta_1+2} \over (\beta_1 - \beta_0 - 1)\zeta(\beta_1)}
\\
\rule{0mm}{7mm}
\delta_0^{-\beta_2+\beta_0+1}
\dm{k^{-\beta_2+2} \over (\beta_2 - \beta_0 - 1)\zeta(\beta_2)}
\end{array}
\right)
\nonumber
\\
&=& 
\left(
\begin{array}{c}
C_1 k^{-\beta_1+2}
\\
C_2 k^{-\beta_2+2}
\end{array}
\right),
\label{ineqn-151205-03}
\end{eqnarray}
where the last equality follows from (\ref{defn-C_1-C_2}). Finally,
(\ref{ineqn-151205-03}) and (\ref{eqn-K}) imply that
(\ref{ineqn-151205-02}) holds for all $ k \ge K+1$.

\medskip
{\bf Acknowledgments.}
The author acknowledges stimulating discussions with Shusaku Sakaiya.

%
%
%

\end{document}